\documentclass[11pt]{article}
\usepackage[numbers,sort&compress]{natbib}
\usepackage{appendix}
\usepackage{enumerate}
\usepackage{amscd}
\usepackage{amsmath}
\usepackage{latexsym}
\usepackage{amsfonts}
\usepackage{amssymb}
\usepackage{amsthm}
\usepackage{bm}
\usepackage{verbatim}
\usepackage{mathrsfs}
\usepackage{enumerate}
\usepackage{hyperref}

\theoremstyle{plain}
\theoremstyle{definition}\newtheorem{theorem}{Theorem}[section]
\theoremstyle{plain}
\theoremstyle{plain}
\theoremstyle{plain}
\theoremstyle{remark}\newtheorem{remark}{Remark}[section]
\usepackage{xcolor}

\newcommand{\Div}{\mathrm{div}\,}
\newcommand{\bv}{{\bm v}}
\newcommand{\bx}{{\bm x}}

\newcommand{\varep}{\varepsilon}

\newcommand{\bl}{{\mbox{\boldmath $\ell$}}}
\newcommand{\bn}{\hat{{\mbox{\boldmath $\ell$}}}}
\newcommand{\bomega}{{\mbox{\boldmath $\omega$}}}
\newcommand{\B}{\Big}

\newcommand{\s}{\mathrm{div}}

\newcommand{\be}{\begin{equation}}
\newcommand{\ee}{\end{equation}}
 \newcommand{\ba}{\begin{aligned}}
 \newcommand{\ea}{\end{aligned}}

  \newcommand{\f}{\frac}
    
  \newcommand{\ben}{\begin{enumerate}}
   \newcommand{\een}{\end{enumerate}}

\newcommand{\Rmnum}[1]{\expandafter\@slowromancap\romannumeral #1@}

\allowdisplaybreaks

\numberwithin{equation}{section}
\begin{document}
%%%%%%%%%%%%%%%%%%%%%%%%%%%%%%%%%%%%%%%%%%%%%%%%%%%%%%%%%%%%%%%%%%%%%%%%%%%%%%%%%%%%%%%%%%%%%%%%%%%%
\title{%Kolmogorov's
Four-fifths   laws   in     incompressible and magnetized fluids: Helicity, Energy and Cross-helicity }
\author{Yulin Ye\footnote{School of Mathematics and Statistics,
		Henan University,
		Kaifeng, 475004,
		P. R. China. Email: ylye@vip.henu.edu.cn},~   ~Yanqing Wang\footnote{  Corresponding author.  College of Mathematics and   Information Science, Zhengzhou University of Light Industry, Zhengzhou, Henan  450002,  P. R. China Email: wangyanqing20056@gmail.com} ~   and  Otto Chkhetiani\footnote{A. M. Obukhov Institute of Atmospheric Physics, Russian Academy of Sciences, Pyzhevsky per. 3, Moscow, 119017, Russia.  Email: ochkheti@ifaran.ru}  }
\date{}
\maketitle
\begin{abstract}
In this paper, we are concerned with  the Kolmogorov's scaling laws of conserved quantities. By means of Eyink's  longitudinal structure functions and the analysis of  interaction of different physical quantities, we extend celebrated
four-fifths   laws  from energy  to helicity in incompressible fluid and, energy and cross-helicity in magnetohydrodynamic
flow. In   contrast to pervious
  4/5 laws of energy and cross-helicity  in magnetized fluids obtained by Politano and  Pouquet, they are in terms of  the mixed  three-order structure functions
rather than the     structure  coupling correlation functions.

  \end{abstract}
\noindent {\bf MSC(2020):}\quad 76F02, 35Q31, 76B03, 76W05 \\\noindent
{\bf Keywords:} 4/5  law;  helicity; energy; cross-helicity; incompressible magnetohydrodynamic fluid  %%%%%%%%%%
\section{Introduction}
\label{intro}
\setcounter{section}{1}\setcounter{equation}{0}
The fluid contribution of
 Kolmogorov in \cite{[Kolmogorov1],[Kolmogorov2],[Kolmogorov3]}    (the so-called K41 theory)   is a milestone in the context of  turbulence. In a  seminal paper \cite{[Kolmogorov3]}, under the homogeneity, isotropy and the finiteness of the energy dissipation, Kolmogorov presented the following four-fifths law in terms of third-order longitudinal structure function
\be\label{Kolmogorov45law}
\langle[\delta \bv_{L}(\bm{r})]^{3} \rangle=-\f45\epsilon \bm{r},
 \ee
 where $\epsilon$ is the mean rate of
  kinetic energy dissipation per unit mass of the Navier-Stokes equations with sufficiently high Reynolds numbers. Here, $\delta \bv_{L}(\bm{r})=\delta \bv (\bm{r})\cdot\f{\bm{r}}{|\bm{r}|}=[\bv(x+\bm{r})-\bv(x)]\cdot\f{\bm{r}}{|\bm{r}|}$ stands for the longitudinal velocity increment and $\langle\cdot\rangle$ denotes the mean value.
  The derivation of \eqref{Kolmogorov45law} in \cite{[Kolmogorov3]}
 relies on the famous  K\'arm\'arth-Howarth equations in \cite{[KH]}.
  The
   Kolmogorov' 4/5 law \eqref{Kolmogorov45law}  has attracted considerable attention   in  turbulence  theory, experiment and numerical simulation (see e.g. \cite{[Chkhetiani2],[Frisch],[KFD],[Eyink4],[Eyink1],[Galtier2],[AB],[AOAZ],[YRS],
  [BCPW],[Drivas],[HPZZ],[MY],[Novack],[Lindborg]}).
The modern rigorous  derivation of
Kolmogorov' 4/5 law \eqref{Kolmogorov45law} can be found in \cite{[Frisch],[Eyink4],[Eyink1]}. In particular,
Eyink   derived the
  local
 version of Kolmogorov's  $4/5$ law \eqref{Kolmogorov45law} below without assumptions of homogeneity and isotropy,
 \be
 S_{L}(\bm{v})=-\f45D(\bm{v}),
 \ee
where $$S_{L}(\bm{v})=\lim\limits_{\lambda\rightarrow0}S_{L}(\bm{v},\lambda)=
\lim\limits_{\lambda\rightarrow0}\f{1}{\lambda}\int_{\partial B } \bl \cdot\delta \bm{v}(\lambda\bl) [\delta \bm{v}_{L}(\lambda\bl)]^{2}\f{d\sigma(\bl) }{4\pi}$$
and
$$
D(\bv)=\lim\limits_{\varepsilon\rightarrow0}\f14
\int_{\mathbb{T}^{3}}\nabla\varphi_{\varepsilon}(\bl)\cdot\delta \bm{v}(\bl)[\delta \bm{v}_{L}(\bl)]^{2}d\bl,$$
in \cite{[Eyink1]}. Here, $\sigma(x)$ stands for the surface measure on the sphere $\partial B=\{x\in \mathbb{R}^{3}: |x|=1\}$ and $\varphi$ is  some smooth non-negative function  supported in $\mathbb{T}^{3}$ with unit integral and $\varphi_{\varepsilon}(x)=\varepsilon^{-3}\varphi(\f{x}{\varepsilon})$.
To this end, Eyink defined the
 longitudinal and transverse velocity   increments as follows
\be\label{eyinksp}
\ba
&\delta \bv_{L}(\bx, t,\bl)=(\bn\otimes \bn)\delta \bv(\bx,t,\bl),~~~~\delta \bv_{T}(\bx, t,\bl)=({\bf 1}-\bn\otimes \bn)\delta \bv(\bx,t,\bl),
\ea\ee
 and invoked the dissipation term $D(v)$  introduced  by
Duchon and Robert in \cite{[DR]}. It is worth pointing out that Duchon and Robert also deduced the following  4/3 law
$$
S(\bv)=-\f43D(\bv)
$$
for the energy of the Euler equations, which corresponds to
\be\label{Yaglom}
\langle \delta \bv_{L}(\bm{r})[\delta \bv (\bm{r})]^{2}  \rangle=-\f43\epsilon \bm{r}.
\ee
This kind  of four-thirds law  originated from Yaglom's work \cite{[Yaglom],[MY]} and was developed in \cite{[MY]}.  The dissipation term $D(\bv)$ was applied   by Duchon and Robert to
study the  Onsager conjecture concerning the minimal regularity of weak solutions for keeping energy conservation of the Euler equations in \cite{[DR]}. In this direction, regarding the energy conservation of weak solutions for the turbulence models, we refer the readers to \cite{[BGSTW],[CCFS],[CET],[DS1],[NV],[DS2],[DR],
[Eyink0],[FW],[GKN],[Isett],[WWWY]}.

% It should be noted that this result actually  corresponds to the Yaglom  $4/3$ law
 %  \be\label{MY4/3law}
%\langle\delta \bv_{L}({\bf r})[\delta \bv({\bf r})]^{2} \rangle=-\f43\epsilon r,
 %\ee
 %where $S(\bv)=-\lim\limits_{\lambda\rightarrow0}S (\bv,\lambda)=-
%\lim\limits_{\lambda\rightarrow0}\f{1}{\lambda}\int_{\partial B }\bl\cdot\delta \bv(\lambda\bl)[\delta \bv(\lambda\bl)]^{2}\f{d\sigma(\bl) }{4\pi}$. Moreover, the minimal regularity of weak solutions for keeping energy conservation of the Euler equations was also studied by Duchon and Robert in \cite{[DR]}, which is connected to the positive part of  Onsager conjecture. In this direction, regarding the energy conservation of weak solutions for turbulence flow, we refer the readers to [].

The character of turbulence rests on not only the non-linear cascades of energy, but also the conserved quantities in three dimensions (see \cite{[AB]}). As known to all, besides the energy $\f12\int_{\mathbb{T}^3} |\bv|^2 dx$, the helicity $ \f12\int_{\mathbb{T}^3} {\bm \omega}\cdot {\bm v} ~dx$ is another
 quadratic conserved   quantity, where the vorticity ${\bm \omega}=\nabla\times {\bm v}$. The  concept of  helicity was first proposed  by   Moffatt  in \cite{[Moffatt]},  which is important at a fundamental level in relation to flow kinematics because it
admits topological interpretation in relation to the linkage or linkages of vortex lines of the flow.   The interaction of the transfer of energy and helicity  in incompressible flow was considered by Chen, Chen and Eyink in \cite{[CCE]}, and, Chen, Chen, Eyink and Holm in \cite{[CCEH]}. For this second invariant quantity, its
  first exact law   was due to Chkhetiani in \cite{[Chkhetiani]}, where the derivation yielded the so-called 2/15 law
\be\label{1.7}
   \langle  \delta \bv_{L}(\bx+{\bm{r}})(\bv_{T}(\bx+{\bm{r}})\times \bv_{T}(\bx) )\rangle =\f{2}{15}h \bm{r}^{2}\epsilon_{H},~~~ \bv_{T}=\bv- \bv_{L},
\ee
 where $\epsilon_{H}$ represents the helicity dissipation rates.
For the 2/15-law \eqref{1.7}, see also  \cite{[Kurien],[KTM]}. Subsequently,
Gomez, Plitano and Pouquet \cite{[GPP]} presented
the following Yaglom type four-thirds law of helicity
via the third-order structure function
\be\label{1.20}
   \langle\delta \bv_{L}({\bf r})[\delta \bv({\bm{r}})\cdot\delta \bomega({\bm{r}})| \rangle -\f12\langle\delta \bomega_{L}({\bm{r}})[\delta \bv({\bm{r}}) ]^{2}  \rangle=-\f43\epsilon_{H} \bm{r}.
\ee  However, to the best of our knowledge, there is no  analogues 4/5 law for helicity in Navier-Stokes or Euler equations. Hence,   a natural question arises that weather or not the Kolmogorov's 4/5  law \eqref{Kolmogorov45law} for the energy can be extended to
 helicity in incompressible fluid. This is our main motivation of the current paper.  Actually, compared with the  4/5 law for
   the energy
   based on the velocity only, it seems that this type law for the helicity requires  the mixed moments of the velocity  and vorticity. This destroys    the symmetry in some sense, consequently, it becomes more complicated to derive
  4/5 law for the
 velocity coupling other quantity  (see \eqref{key identity}).   It also occurs  in magnetized fluid (see \cite[p. 148]{[Biskamp]}), where there exist the nonlinear terms based on the velocity field and the   magnetic field.
In what follows, we invoke Eyink's decomposition \eqref{eyinksp}
  that one splits  the  quantity  into longitudinal   and transverse
parts.
Firstly,  we state our first result on scaling laws for helicity in incompressible fluids.
\begin{theorem}\label{the01}
 Suppose that the pair $(\bv, \bomega)$ satisfy
  the following 3D Euler  equations
\be\label{NS}\left\{\begin{aligned}
	&\bv_t+\s (\bv\otimes \bv)+\nabla P=0,\\
	&\bomega_t+\bv\cdot\nabla \bomega -\bomega\cdot\nabla \bv=0,\\
	&\s \bv=\s \bomega=0,
\end{aligned}\right.
\ee
where $\bv$ and $P$ stand for the velocity and pressure of the
fluid, respectively.
Then, there hold the following local longitudinal and transverse K\'arm\'arth-Howarth equations for helicity, respectively
\be\label{a2}\ba
&\partial_t(\bomega\cdot \bv_L^\varep+\bv\cdot \bomega_L^\varep)+\s (\bomega P_L^\varep+P\bomega_L^\varep) +\s \B[\bv(\bomega\cdot\bv_L^\varep)-\bomega(\bv\cdot\bv_L^\varep)+\bv(\bv\cdot\bomega_L^\varep)\B]\\
&+\s \B[\big(\bv(\bomega_L\cdot \bv_L)\big)^\varep-\bv(\bomega_L\cdot \bv_L)^\varep\B]+\f12 \s\B[\bomega (\bv_L\cdot \bv_L)^\varep-\big(\bomega(\bv_L\cdot\bv_L)\big)^\varep\B]\\
=&-\f43 D_{HL}^\varep(\bv,\bomega),
\ea\ee
and
\be\label{a3}\ba\partial_t(\bomega\cdot &\bv_T^\varep+\bv\cdot \bomega_T^\varep)+\s (\bomega P_T^\varep +P\bomega_T^\varep) +\s\B[\bv (\bomega\cdot \bv_T^\varep)+\bv(\bv\cdot\bomega_T^\varep)-\bomega(\bv\cdot\bv_T^\varep)\B]\\
&+\s\B[\big(\bv(\bv_T\cdot\bomega_T)\big)^\varep-\bv(\bv_T\cdot\bomega_T)^\varep\B]
+\f12\s\B[\bomega(\bv_T\cdot\bv_T)^\varep-\big(\bomega(\bv_T\cdot\bv_T)\big)^\varep\B]\\
=&-\f83D_{HT}^\varep(\bv,\bomega),
\ea\ee
where the dissipation terms (K\'arm\'an-Howarth-Monin type relation) are defined by
\begin{equation}\label{a4}\begin{aligned}
&D_{HL}^\varep(\bv,\bomega)\\
=&\f34\int_{\mathbb{T}^{3}} \nabla \varphi^\varep(\ell)\cdot \delta \bv(\delta \bv_L \cdot \delta \bomega_L)+\f{2}{|\bl|}\varphi^\varep(\ell)\bn\cdot \B[\delta \bv(\delta \bv_T\cdot \bomega_T)+\delta \bv\times (\delta \bomega \times \delta \bv)\B] d^3 \bl\\
&-\f{3}{8} \int_{\mathbb{T}^{3}} \nabla \varphi^\varep(\ell)\cdot \delta \bomega [\delta \bv_{L}]^2+\f{2}{|\bl|}\varphi^\varep(\ell)\bn\cdot \delta \bomega [\delta \bv_T]^2d^3\bl,
\end{aligned}\end{equation}
and
\begin{equation}\label{a5}\begin{aligned}
		&D^\varepsilon_{HT}(\bv,\bomega)\\
		=&\f38\int\nabla \varphi^\varep(\ell)\cdot \delta\bv(\delta \bv_T\cdot \delta \bomega_T)-\f{2}{|\bl|}\varphi^\varep(\ell)\bn\cdot\B[\delta \bv(\delta \bv_T\cdot \delta \bomega_T)+\delta \bv\times (\delta \bomega \times \delta \bv)\B]d^3\bl\\
		&-\f{3}{16}\int \nabla \varphi^\varep\cdot \delta \bomega[\delta \bv_T]^2-\f{2}{|\bl|}\varphi^\varep(\ell)\bn\cdot\delta \bomega[\delta \bv_T]^2 d^3\bl.
\end{aligned}\end{equation} 	
In addition, if suppose that for any $1<m,n<\infty, 3\leq p,q<\infty$ with $\f{2}{p}+\f{1}{m}=1$ and $\f{2}{q}+\f{1}{n}=1$ such that $(\bv, \bomega)$ satisfies
	\be\label{a8}
	\bv\in L^p(0,T;L^q(\mathbb{T}^3))~\text{and}~\bomega\in L^m(0,T;L^n(\mathbb{T}^3)).
	\ee	
	Then the function $D_{HX}^\varep(\bv,\bomega)$ with $X=L~\text{or}~T$ converges to a distribution $D_{H}(\bv,\bomega)$ in the sense of distributions as $\varep\to 0$, and $D_{H}(\bv,\bomega)$ satisfies the local equation of helicity
	\be\label{a9}
	\f12\partial_t (\bv\cdot\bomega)+\f12\s \B[\bv(\bv\cdot\bomega)-\f{1}{2}\bomega(|\bv]^2)+P\bomega\B]=-D_{H}(\bv,\bomega),
	\ee
Furthermore, we have the following 4/5 exact relation for the helicity
\be\label{45helicity}
S_{HL} (\bv,\bomega)= -\f45 D_{H}(\bv,\bomega),
\ee
and 8/15 law
$$S_{HT}(\bv,\bomega)= -\f{8}{15} D_{H}(\bv,\bomega),
$$
where $$S_{HX}(\bv,\bomega)= \lim\limits_{\lambda\to 0}S_{HX}(\bv,\bomega,\lambda),~\text{with}~X=L,T,$$
and
$$\ba
S_{HL}(\bm{v},\bomega,\lambda )
=  &
\f{1}{\lambda}\int_{\partial B } \bn \cdot \B[  \delta \bv(\lambda\bl)\B[\delta \bv_L(\lambda\bl)\cdot \delta \bomega_L(\lambda\bl)\B]-\f12\delta \bomega(\lambda\bl)[\delta \bv_L(\lambda\bl)]^2\B]\f{d\sigma(\bl) }{4\pi}
\\&-\f25\f{1}{\lambda}\int_{\partial B } \bn \cdot \B[\delta \bomega(\lambda\bl) \B[\delta \bv(\lambda\bl)\cdot \delta \bv(\lambda\bl)\B] -\delta \bv(\lambda\bl) \B(\delta \bomega(\lambda\bl) \cdot \delta \bv(\lambda\bl)\B)\B] \f{d\sigma(\bl) }{4\pi},\\
S_{HT}(\bm{v},\bomega,\lambda ) =&
\f{1}{\lambda}\int_{\partial B } \bn \cdot \B[  \delta \bv(\lambda\bl)\B[\delta \bv_T(\lambda\bl)\cdot \delta \bomega_T(\lambda\bl)\B]-\f12\delta \bomega(\lambda\bl)[\delta \bv_T(\lambda\bl)]^2\B] \f{d\sigma(\bl) }{4\pi} \\&+\f25\f{1}{\lambda}\int_{\partial B } \bn \cdot \B[\delta \bomega(\lambda\bl) \B(\delta \bv(\lambda\bl)\cdot \delta \bv(\lambda\bl)\B) -\delta \bv(\lambda\bl) \B(\delta \bomega(\lambda\bl) \cdot \delta \bv(\lambda\bl)\B)\B] \f{d\sigma(\bl) }{4\pi}.
\ea$$
\end{theorem}
\begin{remark}
It is noted that the 4/5 law  \eqref{45helicity} corresponds to the
   following four-fifths law for longitudinal third-order moments
\be\label{1.15}
\langle \delta \bv_{L} (\delta \bv_{L} \cdot\delta \bomega_{L} )  \rangle- \f12\langle \delta \bomega_{L} (\delta \bv_{L} )^{ 2}  \rangle  -\f25\langle  \delta\bomega_{L}(\delta\bv)^{2}\rangle + \f25\langle  \delta\bv_{L}(\delta\bv\cdot\delta\bomega)\rangle =-\f45\epsilon_{H} \bm{r}.
\ee

\end{remark}\begin{remark}
	The 4/5 law and 8/15 law for helicity in Theorem \ref{the01} rely on the assumption that the limits of $S_{HX}(\bv,\bomega, \lambda)$ with $X=L,T$ exist as $\lambda \to 0$. Actually, this assumption can be removed by following the strategy due to  Novack \cite{[Novack]}, in which the author considered the   validity of certain exact results from turbulence theory in
	the deterministic setting. We leave this to the interested reader.
\end{remark}
\begin{remark}
A well known helicity criterion for the incompressible Euler equations obtained in \cite{[CCFS]} is that the helicity of weak solutions  is conserved provided that  $v\in L^3(0,T;B^{\f23}_{3,q^{\natural}})$ with $1\leq q^{\natural}<\infty.$ Hence, if $m\geq3$ in \eqref{a8}, we require $n<9/4$ in this theorem. Since a special case of \eqref{a8} is $p=m=3$, $q=\f92$ and $n=\f95$, the condition \eqref{a8} is  no empty.
\end{remark}
\begin{remark}\label{rem1.4}
Thanks to the  vector triple product formula
$$
\bm{A}\times(\bm{B}\times \bm{C})=\bm{B}(\bm{A}\cdot\bm{C})-\bm{C}(\bm{A}\cdot \bm{B}),~\text{for\ any\ vectors\ }\bm{A},\bm{B}~\text{and}\ \bm{C},
$$
one may reformulate $S_{HX}(\bv,\bomega)$ by  $\delta \bv\times (\delta \bomega \times \delta \bv)$ in this theorem.
\end{remark}
 The new ingredient   in this theorem is two-fold. First, we prove new exact laws   and generalize  famous Kolmogorov's 4/5 law  \eqref{Kolmogorov45law} from energy to helicity for the incompressible flow. Second, we present the effect of mixed moments in four-fifths laws
  in terms of longitudinal third order.   The interaction of different quantities brings more difficulty to derive a scaling relation analogous to Kolmogorov's 4/5 law (see \cite[p. 148]{[Biskamp]}). In Eyink's framework, it seems that two  identities \eqref{c18} and \eqref{c19} in \cite{[Eyink1]}   break  down for the velocity coupling with other different quantities.   To over this difficulty, we are trying to find the accurate relation between $\f{\partial}{\partial_{\ell_k}}\B(\bn_i\bn_j\B)$ and $\B(\f{\partial \bn_i}{\partial\ell_j}+\f{\partial \bn_j}{\partial\ell_i}\B)\bn_k$ involving different variable quantities (see \eqref{2.27}), which  are equivalent in deriving the 4/5 law of energy just involving the velocity field quantity in the incompressible fluids.
To this end, we  deduce two modified    identities
\eqref{2.26} and \eqref{2.44}  allowing different quantities, which reduce to Eyink's identities when the different quantities are the same.  The starting point is that the following new identity
\be\label{key identity}\ba
&\B[\f{\partial}{\partial_{\ell_k}}\B(\bn_i\bn_j\B)-\B(\f{\partial \bn_i}{\partial\ell_j}+\f{\partial \bn_j}{\partial\ell_i}\B)\bn_k\B]A_kB_iC_j\\
=&\f{1}{|\bl|}\bn\cdot\B[{\bf C}({\bf A}\cdot {\bf B})+{\bf B}({\bf A}\cdot {\bf C})-2{\bf A}({\bf B}\cdot {\bf C})\B]
\ea\ee
 is observed, which immediately becomes zero when ${\bf A}={\bf B}= {\bf C}.$

Furthermore, in light of the key observation \eqref{key identity}, we can deal with other scaling laws involving different variable quantities. Next, we turn our attention to the incompressible magnetohydrodynamic
fluid, in which there exist the nonlinear terms involving both the
the
velocity $\bv$ and the magnetic field ${\bm h}$. Two 4/5 laws for total energy and cross-helicity obtained by    Politano and  Pouquet in the  magnetized fluids  in \cite{[PP2]} can be written as
\be\label{pp1}\ba
&\langle[\delta \bv_{L}({\bm r})]^{3} \rangle- 6\langle {\bm h}^{2}_{L}\delta \bv_{L}({\bm r})  \rangle=-\f45\epsilon_{N} \bm{r},\\
&\langle[\delta {\bm h }_{L}({\bm r})]^{3} \rangle- 6\langle {\bm h}^{2}_{L}\delta \bv_{T}({\bm r})  \rangle=-\f45\epsilon_{C} \bm{r},
\ea
\ee
where $\epsilon_{N}$ and $\epsilon_{C}$ stand for total energy  and cross-helicity dissipation rates, respectively (see also \cite{[Chandrasekhar]}). In another contemporary paper \cite{[PP1]},  Politano and  Pouquet  established the Yaglom scaling laws for energy and cross-helicity. Indeed, it is very useful to invoke these scaling laws to model the energy flux in turbulent plasmas. We refer the reader to  \cite{[MS]} for a  detailed review of scaling laws for the energy transfer in space plasma turbulence  recently provided by Marino and Sorriso-Valvo. However, notice    that the second term    in $\eqref{pp1}_{1}$  and $\eqref{pp1}_2$ are  $\langle {\bm h}^{2}_{L}\delta \bv_{L}({\bf r})  \rangle$ and $\langle {\bm h}^{2}_{L}\delta \bv_{T}({\bf r})  \rangle$  based on  the
structure and correlation functions rather than  $\langle\delta {\bm h}^{2}_{L}\delta \bv_{L}({\bf r})  \rangle$  and $\langle\delta {\bm h}^{2}_{L}\delta \bv_{T}({\bf r})  \rangle$ via the
structure functions only. Hence, our second objective is to revisit
Kolmogorov type scaling laws \eqref{pp1} in plasma turbulence.
In the spirit of above theorem, we formulate 4/5 laws of total
energy and cross-helicity in the inviscid  magnetohydrodynamic flow.
\begin{theorem}\label{the1.2}
 Suppose that the   triplet $(\bm{v},\bm{h}, \Pi)$  satisfy the following
\be\label{MHD}\left\{\ba
&\partial_{t} \bm{v}  +\text{div} (\bm{v}\otimes \bm{v} )  -\text{div} (\bm{h}\otimes \bm{h} )  + \nabla\Pi =0, \\
&\partial_{t} \bm{h} +\text{div} (\bm{v}\otimes \bm{h} )   -\text{div} (\bm{ h}\otimes \bm{ u} ) =0, \\
&\Div \bm{v} =\Div \bm{h} =0,
 \ea\right.\ee
where $\bv$, $\bm{h}$ and $\Pi$ stand for the veolcity field, magnetic field and the pressure of the fluid, respectively. Then, there hold the following local longitudinal and transverse K\'arm\'arth-Howarth equations for energy,
  $$\ba
 &\partial_{t}( \bm{v}_{L}^{\varepsilon}\cdot\bm{v} +\bm{h}\cdot\bm{h}_{L}^{\varepsilon})+\s\B[\bm{v}( \bm{v}\cdot \bm{v}_{L }^{\varepsilon})-\bm{h}(\bm{h} \cdot \bm{v}_{L }^{\varepsilon})
+\bm{v}(\bm{h} \cdot \bm{h}_{L }^{\varepsilon})
-\bm{h}(\bm{v} \cdot \bm{h}_{L }^{\varepsilon})\B]
\\&+\s\B[\Pi_L^\varep \bv+\Pi \bv_L^\varep\B]
  -\s\B[\big(\bm{h} (\bm{v}_L\cdot \bm{h}_L)\big)^\varep-\bm{h}(\bm{v}_L\cdot \bm{h}_L)^\varep\B]\\& +\f12\s \B[\big(\bm{v}(\bm{v}_L\cdot\bm{v}_L)\big)^\varep-\bm{v}(\bm{v}_L\cdot\bm{v}_L)^\varep\B]+\f12
\s \B[\big(\bm{v}(\bm{h}_L\cdot\bm{h}_L)\big)^\varep
-\bm{v}(\bm{h}_L\cdot\bm{h}_L)^\varep\B]\\
  = & -\f{2}{3}D^{\varepsilon}_{EL}(\bm{v},\bm{h});
\\ &\partial_{t}( \bm{v}_{T}^{\varepsilon}\cdot\bm{v} +\bm{h}\cdot\bm{h}_{T}^{\varepsilon})+\s\B[\bm{v}( \bm{v}\cdot \bm{v}_{T }^{\varepsilon})
-\bm{h}(\bm{h} \cdot \bm{v}_{T }^{\varepsilon})
+\bm{v}(\bm{h} \cdot \bm{h}_{T }^{\varepsilon})
-\bm{h}(\bm{v} \cdot \bm{h}_{T }^{\varepsilon})\B]
\\&+\s\B[\Pi_T^\varep \bv+\Pi \bv_T^\varep\B]
 -\text{div} \B[\big(\bm{h}(\bm{v}_T\cdot \bm{h}_T)\big)^\varep-\bm{h}(\bm{v}_T\cdot \bm{h}_T)^\varep\B]\\& +\f12\s \B[\big(\bm{v}(\bm{v}_T\cdot\bm{v}_T)\big)^\varep-\bm{v}(\bm{v}_T\cdot\bm{v}_T)^\varep\B]+\f12\s \B[\big(\bm{v}(\bm{h}_T\cdot\bm{h}_T)\big)^\varep
-\bm{v}(\bm{h}_T\cdot\bm{h}_T)^\varep\B] \\
=&-\f43D_{ET}^{\varepsilon}(\bm{v},\bm{h});
   \ea$$
   and the following local longitudinal and transverse K\'arm\'arth-Howarth equations for cross-helicity
   $$\ba& \partial_{t} (\bm{v}_{L}^{\varepsilon}\cdot\bm{h}+\bm{v}\cdot\bm{h}_{L}^{\varepsilon} )+\s\B[\bm{v}(\bm{v}\cdot \bm{h}_{L }^{\varepsilon})
  -\bm{h} (\bm{h} \cdot \bm{h}_{L }^{\varepsilon})
  +\bm{v} (\bm{h} \cdot \bm{v}_{L }^{\varepsilon})
  -\bm{h} (\bm{v} \cdot \bm{v}_{L }^{\varepsilon})\B]\\
&+\s(\Pi^\varep_L \bm{h}+\bm{h}_L^\varep \Pi) +\s\B[\big(\bv (\bm{h}_L\cdot \bm{v}_L)\big)^\varep-\bv(\bm{h}_L\cdot \bm{v}_L)^\varep\B]
\\& -\f12\s \B[\big(\bm{h}(\bm{h}_L\cdot\bm{h}_L)\big)^\varep-\bm{h}(\bm{h}_L\cdot\bm{h}_L)^\varep\B]-\f12
\s \B[\big(\bm{h}(\bm{v}_L\cdot\bm{v}_L)\big)^\varep-\bm{h}(\bm{v}_L\cdot\bm{v}_L)^\varep\B]\\
   =&-\f23D^{\varepsilon}_{CHL}(\bm{v},\bm{h});
\\ & \partial_{t} (\bm{v}_{T}^{\varepsilon}\cdot\bm{h}+\bm{v}\cdot\bm{h}_{T}^{\varepsilon} )+\s\B[\bm{v}(\bm{v}\cdot \bm{h}_{T }^{\varepsilon})
-\bm{h} (\bm{h} \cdot \bm{h}_{T }^{\varepsilon})
+\bm{v} (\bm{h} \cdot \bm{v}_{T }^{\varepsilon})
-\bm{h} (\bm{v} \cdot \bm{v}_{T }^{\varepsilon})\B]\\
& +\s(\Pi^\varep_T \bm{h}+\bm{h}_T^\varep \Pi)+\s\B[\big(\bv (\bm{h}_T\cdot \bm{v}_T)\big)^\varep-\bv(\bm{h}_T\cdot \bm{v}_T)^\varep\B]
\\& -\f12\s \B[\big(\bm{h}(\bm{h}_T\cdot\bm{h}_T)\big)^\varep-\bm{h}(\bm{h}_T\cdot\bm{h}_T)^\varep\B]-\f12
\s \B[\big(\bm{h}(\bm{v}_T\cdot\bm{v}_T)\big)^\varep-\bm{h}(\bm{v}_T\cdot\bm{v}_T)^\varep\B]\\
=&-\f43D^{\varepsilon}_{CHT}(\bm{v},\bm{h}).
    \ea $$
Here,
\begin{align}
D^{\varepsilon}_{EL}(\bv,\bm{h})=&\f34\int_{\mathbb{T}^{3}} \nabla \varphi^\varep(\ell)\cdot \delta \bm{v} \B([\delta \bm{v}_L]^2+[\delta\bm{h}_L]^2\B)+\f{2}{|\bl|}\varphi^\varep(\ell)\bn\cdot \delta \bm{v} \B([\delta \bm{v}_T]^2+[\delta\bm{h}_T]^2\B) d^3\bl
\nonumber\\
-&\f{3}{2}\int_{\mathbb{T}^{3}} \nabla \varphi^\varep(\ell)\cdot \delta \bm{h}(\delta \bm{v}_L \cdot \delta \bm{h}_L)+\f{2}{|\bl|}\varphi^\varep(\ell)\bn\cdot \delta \bm{h}(\delta \bm{v}_T\cdot \delta\bm{h}_T)d^3\bl \nonumber\\
-&3 \int_{\mathbb{T}^{3}} \f{1}{|\bl|}\varphi^\varep(\ell)\bn\cdot\B[\delta \bm{h}\times(\delta \bm{v} \times\delta \bm{h}) \B] d^3 \bl;\nonumber\\
D_{ET}^{\varepsilon}(\bv,\bm{h})=  &  \f38\int_{\mathbb{T}^{3}} \nabla \varphi^\varep(\ell)\cdot \delta \bm{v}\B([\delta \bm{v}_T]^2+[\delta \bm{h}_T]^2\B)-\f{2}{|\bl|}\varphi^\varep\bn\cdot \delta \bm{v} \B([\delta \bm{v}_T]^2 +[\delta \bm{h}_T]^2\B)  d^3\bl\nonumber\\  -&\f34\int_{\mathbb{T}^{3}} \nabla \varphi^\varep(\ell)\cdot \delta \bm{h}(\delta \bm{v}_T\cdot \delta \bm{h}_T)-\f{2}{|\bl|}\varphi^\varep(\ell)\bn\cdot \delta\bm{h} (\delta\bm{v}_T\cdot \delta\bm{h}_T)d^3\bl\nonumber\\+&\f32\int_{\mathbb{T}^{3}}\f{1}{|\bl|}\varphi^\varep(\ell)\bn\cdot
[\delta \bm{h}\times(\delta\bm{v}\times \delta\bm{h})]
d^3\bl;
\nonumber\\
 D^{\varepsilon}_{CHL}(\bv,\bm{h})=&
\f32\int_{\mathbb{T}^{3}} \nabla \varphi^\varep(\ell)\cdot \delta \bm{v}(\delta \bm{h}_L \cdot \delta \bm{v}_L)+\f{2}{|\bl|}\varphi^\varep(\ell)\bn\cdot \delta \bm{v}(\delta \bm{h}_T\cdot \delta\bm{v}_T) d^3 \bl \nonumber\\-&  \f{3}{4}\int_{\mathbb{T}^{3}} \nabla \varphi^\varep(\ell)\cdot \delta \bm{h} \B([\delta \bm{h}_L]^2+[\delta \bv_L]^2\B)+\f{2}{|\bl|}\varphi^\varep(\ell)\bn\cdot \delta \bm{h} \B([\delta \bm{h}_T]^2+[\delta \bv_T]^2\B) d^3\bl\nonumber\\
+&\f32\int_{\mathbb{T}^3}  \f{2}{|\bl|}\varphi^\varep(\ell)\bn\cdot\B[\delta \bm{v}\times(\delta \bm{h} \times \delta \bm{v})\B] d^3\bl;
\nonumber\\
  D_{CHT}^{\varepsilon} (\bv,\bm{h}) =&\f34\int_{\mathbb{T}^{3}} \nabla \varphi^\varep(\ell)\cdot \delta \bm{v}(\delta \bm{h}_T\cdot \delta \bm{v}_T)-\f{2}{|\bl|}\varphi^\varep\bn\cdot \delta\bm{v} (\delta\bm{h}_T\cdot \delta\bm{v}_T) d^3\bl \nonumber\\
  -& \f{3}{8}\int_{\mathbb{T}^{3}} \nabla \varphi^\varep(\ell)\cdot \delta \bm{h}\B([\delta \bm{h}_T]^2+[\delta \bv_T]^2\B)-\f{2}{|\bl|}\varphi^\varep\bn\cdot \delta \bm{h} \B([\delta \bm{h}_T]^2 +[\delta \bv_T]^2\B)d^3\bl\nonumber\\
  -&\f34\int_{\mathbb{T}^3} \f{2}{|\bl|}\varphi^\varep(\ell)\bn\cdot\B[\delta \bm{v}\times(\delta\bm{h}\times \delta\bm{v}) \B]d^3\bl.\nonumber
 \end{align}
 In addition, if suppose that   $(\bv, \bm{h})$ satisfies
	\be\label{1.19}
	\bv,\bm{h}\in L^3(0,T;L^3(\mathbb{T}^3)).
	\ee	
Then the function $D_{YX}^\varep(\bv,\bm{h})$ with $X=L,T$ and $ Y=E, CH$ converges to a distribution $D_{Y}(\bv,\bm{h})$ in the sense of distributions as $\varep\to 0$, and $D_{Y}(\bv,\bm{h})$ satisfies the local equation of total energy and cross-helicity, respectively
\be \ba\label{1.18}
\f12 \partial_{t}(  |\bv|^{2} + |\bm{h}|^{2} )+\text{div}\B[\f12\bv|\bv|^{2}+\f12\bv|\bm{h}|^{2}-\bm{h}(\bm{h}\cdot\bv)
 +\bv\Pi \B]= - D_{E}( \bv, \bm{h})
 \ea\ee
\be \ba	\partial_{t}(  \bv\cdot \bm{h})+\text{div}\B[  \bv(\bv\cdot \bm{h})-\f12(|\bm{h}|^{2})\bm{h}-\f12\bm{h}|\bv|^{2}+  \bm{h}\Pi\B]= -D_{CH }( \bv, \bm{h}).
\ea\ee
Furthermore, we have the following 4/5   and 8/15 laws
\begin{align} \label{mhde45}
&S_{EL}(\bv, \bm{h})=-\f45 D_{E}(\bv, \bm{h}),~ S_{ET}(\bv, \bm{h})=-\f{8}{15}D_{E}(\bv, \bm{h})\\
 \label{mhdc45}&S_{C HL}(\bv, \bm{h})=-\f45 D_{CH}(\bv, \bm{h}),~
 S_{CHT}(\bv, \bm{h})=-\f{8}{15}D_{CH}(\bv, \bm{h}),
\end{align}
where $$S_{YX}(\bv,\bm{h})= \lim\limits_{\lambda\to 0}S_{YX}(\bv,\bm{h},\lambda),~\text{with}~X=L,T, Y=E,CH,$$
and
\begin{align}
& S_{EL}(\bv, \bm{h},\lambda)\nonumber\\
  = &\f{1}{\lambda}\int_{\partial B } \bn \cdot \B[ \delta \bm{v}(\lambda\bl)|\B(\delta \bm{v}_L(\lambda\bl)]^2+[\delta \bm{h}_L(\lambda\bl)]^2\B)-2\delta \bm{h}(\bl)\B(\delta \bm{v}_L(\lambda\bl) \cdot \delta \bm{h}_L(\lambda\bl)\B)\B] \f{d\sigma(\bl) }{4\pi}\nonumber\\  &-\f45 \f{1}{\lambda}\int_{\partial B } \bn \cdot \B[ \delta \bm{h}(\lambda\bl)\times\B(\delta \bm{v}(\lambda\bl) \times\delta \bm{h}(\lambda\bl)\B)\B] \f{d\sigma(\bl) }{4\pi};\nonumber\\
 & S_{ET}(\bv, \bm{h},\lambda)\nonumber\\
 =& \f{1}{\lambda}\int_{\partial B } \bn \cdot \B[\delta \bm{v}(\lambda\bl)\B( [\delta \bm{v}_T(\lambda\bl)]^2+ [\delta \bm{h}_T(\lambda \bl)]^2\B)-2 \delta\bm{h}(\lambda\bl) \B(\delta\bm{v}_T(\lambda\bl)\cdot \delta\bm{h}_T(\lambda\bl)\B)\B] \f{d\sigma(\bl) }{4\pi}\nonumber\\~~~~~~~~~\,&  +\f45 \f{1}{\lambda}\int_{\partial B } \bn\cdot \B[\delta \bm{h}(\lambda\bl)\times\B(\delta\bm{v}(\lambda\bl)\times\delta \bm{h}(\lambda\bl)\B)\B] \f{d\sigma(\bl) }{4\pi};\nonumber\\
 &  S_{CHL}(\bv, \bm{h},\lambda)\nonumber\\
   =&
 \int_{\partial B } \bn \cdot \B[ 2\delta \bm{v}(\lambda\bl)\B(\delta \bm{h}_L(\lambda\bl) \cdot \delta \bm{v}_L(\lambda\bl)\B)-\delta \bm{h} (\lambda\bl)\B([\delta \bm{h}_L(\lambda\bl)]^2+[\delta \bm{v}_L(\lambda\bl)]^2\B)\B] \f{d\sigma(\bl) }{4\pi}\nonumber\\~~~~~~~~&  -\f45 \f{1}{\lambda}\int_{\partial B } \bn\cdot \B[\delta \bm{v}(\lambda\bl)\times\B(\delta \bm{h}(\lambda\bl) \times\delta \bm{v}(\lambda\bl)\B)\B] \f{d\sigma(\bl) }{4\pi};\nonumber
\\&  S_{CHT}(\bv, \bm{h},\lambda)\nonumber\\
= & \f{1}{\lambda}\int_{\partial B } \bn \cdot \B[2\delta\bm{v}(\lambda\bl) \B(\delta\bm{h}_T(\lambda\bl)\cdot \delta\bm{v}_T(\lambda\bl)\B)- \delta \bm{h}(\lambda \bl)\B( [\delta \bm{v}_T(\lambda\bl)]^2+[\delta \bm{h}_T(\lambda\bl)]^2\B) \B] \f{d\sigma(\bl) }{4\pi}\nonumber\\& +\f45 \f{1}{\lambda}\int_{\partial B } \bn \cdot \B[\delta \bm{v}(\lambda\bl)\times\B(\delta\bm{h}(\lambda\bl)\times\delta \bm{v}(\lambda\bl)\B)\B] \f{d\sigma(\bl) }{4\pi}.\nonumber
\end{align}
  \end{theorem}
\begin{remark}
According to the    vector triple product formula in Remark \ref{rem1.4}, one could rewrite  $S_{EL}(\bv, \bm{h},\lambda)$, $S_{ET}(\bv, \bm{h},\lambda)$, $S_{CHL}(\bv, \bm{h},\lambda)$  and $S_{CHT}(\bv, \bm{h},\lambda)$ slightly. Hence, the 4/5 laws \eqref{mhde45} and \eqref{mhdc45} correspond  to the
   following 4/5 laws for longitudinal third-order moments
\be\ba\label{1.24}
&\langle \delta \bv_{L}  (\delta \bv_{L})^{2}  \rangle+  \langle \delta \bv_{L} (\delta \bm{h}_{L} )^{ 2}  \rangle  -2\langle \delta \bm{h}_{L} (\delta \bm{h}_{L} \cdot\delta \bv_{L} )
\rangle-\f45\langle \delta \bm{h}_{L}(\delta \bm{h} \cdot\delta \bm{v} )  \rangle+\f45\langle \delta \bm{v}_{L}(\delta \bm{h} ) ^{2}  \rangle=-\f45\epsilon_{E} \bm{r},\\
&2\langle \delta \bv_{L} (\delta \bm{h}_{L} \cdot\delta \bv_{L} )
\rangle- \langle \delta \bm{h}_{L} (\delta \bm{h}_{L} )^{2}   \rangle-\langle \delta \bm{h}_{L}  (\delta \bv_{L})^{2} \rangle-\f45\langle \delta \bm{h}_{L} ( \delta \bv )^{2}
\rangle+\f45\langle \delta \bm{v}_{L} ( \delta \bv\cdot \delta \bm{h}  )
\rangle=-\f45\epsilon_{C} \bm{r}.
\ea\ee
\end{remark}
\begin{remark}
The four-fifths law \eqref{mhde45} concerning  energy in the inviscid MHD system
reduces to the celebrated   Kolmogorov four-fifths  law showed by   Eyink  in \cite{[Eyink1]} when the magnetic field becomes constant.
\end{remark}
\begin{remark} We rewrite the inviscid MHD equations \eqref{MHD} in terms of the
 Els\"asser variances
 ${\bm Z}^{+}=\f{\bm{v}+\bm{h}}{2}$,  ${\bm Z}^{-}=\f{\bm{v}-\bm{h}}{2}$  as follows
\be\label{MHDE}\left\{\ba
&\partial_{t}{\bm Z}^{+} + {\bm Z}^{-}\cdot {\bm Z}^{+}+\nabla \Pi=0, \\
&\partial_{t}{\bm Z}^{-} + {\bm Z}^{+}\cdot {\bm Z}^{-}+\nabla \Pi=0, \\
&\Div {\bm Z}^{+}=\Div {\bm Z}^{-}=0.
 \ea\right.\ee
Following the path of the above theorems, one could show that
\be\ba\label{1.26}
&\langle \delta \bm{ Z}^{-}_{L}  [\delta \bm{Z}^{+}_{L}|^{2}  \rangle-\f25\langle \delta \bm{Z}^{+}_{L} (\bm{Z}^{+}_{L}\cdot \bm{Z}^{-}_{L} )    \rangle+\f25\langle \delta \bm{Z}^{-}_{L}  [\delta \bm{Z}^{+} |^{2}  \rangle=-\f45\epsilon_{+} r,\\
&\langle \delta \bm{Z^{+}}_{L}  [\delta \bm{Z^}{-}_{L}|^{2}  \rangle-\f25\langle \delta \bm{Z}^{-}_{L} (\bm{Z}^{-}_{L}\cdot \bm{Z}^{+}_{L} )    \rangle+\f25\langle \delta \bm{Z}^{+}_{L}  [\delta \bm{Z}^{-} |^{2}  \rangle=-\f45\epsilon_{+} r.
\ea\ee
This is left to the readers.
\end{remark}
\begin{remark}
Compared with the application of the   K\'arm\'arth-Howarth type equations in a statistical view,  the filtering approach for derivation of
the scaling laws  developing in \cite{[DR],[Eyink1]}  holds for individual realizations of fluids. We refer to \cite{[Eyink1],[Eyink4],[KAM],[Dubrulle]} for advantage of this method.
\end{remark}

\begin{remark}Just as the relationship between four-fifths law  \eqref{Kolmogorov45law} and  four-thirds law \eqref{Yaglom}, the left-hand side of \eqref{1.20} also appears to the left of \eqref{1.15},   so does that in \eqref{1.24} and \eqref{1.26}.
\end{remark}
\begin{remark}
It is worth remarking that a four-fifths law of magnetic helicity in  the    electron    magnetohydrodynamic flow was established  in \cite{[Chkhetiani2]}. In a forthcoming
paper, we will address  four-fifths laws for energy and   magnetic helicity in electron and Hall magnetohydrodynamic
fluids.
\end{remark}

The remainder of this paper is organized as follows.
  Section 2 is devoted to  Kolmogorov type exact scaling relation   of    helicity in the incompressible fluids. In Section 3, we study dissipation rates of total energy and cross helicity in the magnetohydrodynamics
flow, respectively.   Finally, a summary  is provided.
\section{Four-fifths   laws of helicity in incompressible fluids }
In this section, we apply Eyink's   decomposition \eqref{eyinksp} involving   the longitudinal and transverse velocity increments   to  consider   the dissipation rates of  helicity in incompressible fluids.
A new identity
\eqref{2.27} is observed, which allows one to deal with the  interaction of different physical quantities.

\begin{proof}[Proof of Theorem \ref{the01}]
	The proof of Theorem \ref{the01} is divided into four steps.
	
	\underline{{\bf Step 1}}: First, inspired by the work of Duchon-Robert \cite{[DR]} and Eyink \cite{[Eyink1]}, we estiblish the local longitudinal and transverse K\'arm\'arth-Howarth equations for the helicity of  Euler equations \eqref{NS}. To do this, we need to consider the longitudinal and transverse balance equations for the Euler equations separately.  We define the longitudinal and transverse velocity (vorticity) relative to a vector $\bl$ as follows:
	\be\label{c1}\ba
	&\bv_L(\bx,t,\bl)=(\bn\otimes \bn)\cdot \bv(\bx+\bl,t),~~\bv_T(\bx,t,\bl)=({\bf 1}-\bn\otimes \bn)\cdot \bv(\bx+\bl,t);\\
		&\bomega_L(\bx,t,\bl)=(\bn\otimes \bn)\cdot \bomega(\bx+\bl,t),~\bomega_T(\bx,t,\bl)=({\bf 1}-\bn\otimes \bn)\cdot \bv(\bx+\bl,t);
	\ea\ee
which implies the fact that $\bv_L+\bv_T=\bv$ and  $\bomega_L+\bomega_T=\bomega$. Let $\varphi(\ell)$ be any $C_0^\infty$ function, nonnegative with unit integral, radially symmetric, and let $\varphi^\varepsilon(\ell)=\f{1}{\varepsilon^3}\varphi(\f{\ell}{\varepsilon})$. Then  the mollified version of $\bv_X, \bomega_X$ with $X=L,T$ can be defined as follows
\be\label{c2}\ba
\bv^\varep_X(\bx,t)=\int_{\mathbb{T}^{3}} \varphi^\varep(\bl)\bv_X(\bx,t;\bl)d^3\bl,~~~X=L,T,
\ea\ee
and
\be\label{c3}\ba
\bomega^\varep_X(\bx,t)=\int_{\mathbb{T}^{3}} \varphi^\varep(\bl)\bomega_X(\bx,t;\bl)d^3\bl,~~~X=L,T,
\ea\ee
respectively.

To obtain the longitudinal velocity (vorticity) equations relative to vector $\bl$,  multiplying $\eqref{NS}_1$ and $\eqref{NS}_2$ by $\varphi^\varep(\bl)(\bn\otimes \bn)$ and integrating over $\mathbb{T}^3$, we have
\be\label{c4}\ba
&\partial_t \bv_L^\varep+\s (\bv\otimes \bv_L)^\varep+\s {\bf \Pi}_L^\varep=0,\\
&\partial_t \bomega_L^\varep+\s (\bv\otimes \bomega_L)^\varep-\s (\bomega \otimes \bv_L)^\varep=0,
\ea\ee
where
\be\label{c5}
{\bf \Pi}^\varep_L=\int_{\mathbb{T}^{3}}  \varphi^\varep(\bl)(\bn\otimes \bn)P(\bx+\bl,t)d^3\bl.
\ee
By the same token, multiplying $\eqref{NS}_1$ and $\eqref{NS}_2$ by $\varphi^\varep(\bl)({\bf 1}-\bn\otimes \bn)$ and integrating over $\mathbb{T}^3$, it follows that
\be\label{c6}\ba
&\partial_t \bv_T^\varep+\s (\bv\otimes \bv_T)^\varep+\s {\bf \Pi}_T^\varep=0,\\
&\partial_t \bomega_T^\varep+\s (\bv\otimes \bomega_T)^\varep-\s (\bomega \otimes \bv_T)^\varep=0,
\ea\ee
where
\be\label{c7}
{\bf \Pi}^\varep_T=\int_{\mathbb{T}^{3}}  \varphi^\varep(\bl)({\bf 1}-\bn\otimes \bn)P(\bx+\bl,t)d^3\bl.
\ee
Before proceeding further, we  claim that
\be\label{c8}
\s {\bf \Pi}^\varep_X(\bx,t)=\nabla P_X(\bx,t),~~X=L,T,
\ee
holds in the sense of distributions, where $P_L(\bx,t), P_T(\bx,t)$ are scalar functions defined  as follows
\be\label{c9}
P_X^\varep(\bx,t)=\int_{\mathbb{T}^{3}}  \varphi^\varep_X(\ell)P(\bx+\bl,t)d^3\bl,~~X=L,T,
\ee
with
\be\label{c10}
\varphi_L(\ell)=\varphi(\ell)-\varphi_T(\ell),~~~~\varphi_T(\ell)=2\int_{|\bl|}^\infty \f{\varphi(\ell^{'})}{\ell^{'}} d\ell^{'}.
\ee
With the help of  the definition of $\varphi(\ell)$, it is easy to see that $\varphi_L(\ell)$ and $ \varphi_T(\ell)$ defined here  are compactly supported and $C^\infty$ everywhere except at $0$, where they have a mild (logarithmic) singularity.

Now we are in a position to  show the validity of \eqref{c8}. On the one hand, using a straightforward computation, we have the following basic relation
\be\label{c11}\ba
\partial_{\ell_k}\bn_i&=\f{\partial}{\partial\ell_k}\B(\f{\ell_i}{|\bl|}\B)=\f{\delta^{ik}|\bl|-\f{\ell_i\ell_k}{|\bl|}}{|\bl]^2}=\f{1}{|\bl|}(\delta^{ik}-\bn_i\bn_k),
\ea\ee
which together with the integration by parts yields
\be\label{c12}\ba
\s {\bf \Pi}^\varep_L&=\partial_{x_i}\int_{\mathbb{T}^{3}}  \varphi^\varep(\ell)(\bn_i\bn_j)P(\bx+\bl) d^3\bl\\
&=\int_{\mathbb{T}^{3}}  \varphi^\varep(\ell)(\bn_i\bn_j)\partial_{x_i}P(\bx+\bl) d^3\bl\\
&=-\int_{\mathbb{T}^{3}}  \B[\f{d\varphi^\varep(\ell)}{d\ell}\bn_j+\varphi^\varep(\ell)\f{2}{|\bl|}\bn_j\B]P(\bx+\bl) d^3\bl.
\ea\ee
 On ther other hand, in light of the  definition of  $P_L^\varep$ and \eqref{c10}, we have
\be\label{c13}\ba
\nabla \varphi^\varep_L(\ell)=\B(\f{d\varphi^\varep}{d\ell}(\ell)+\f{2}{|\bl|}\varphi^\varep(\ell)\B)\bn.
\ea\ee
Combining \eqref{c12}, \eqref{c13} and using the integration by parts again, we can deduce that
\be\label{c14}\ba
\s {\bf \Pi}^\varep_L =-\int_{\mathbb{T}^{3}}  \partial_{\ell_j}\varphi^\varep_L(\ell)P(\bx+\bl,t)d^3\ell 
  =\partial_{x_j}\int_{\mathbb{T}^{3}} \varphi^\varep_L(\ell)P(\bx+\bl,t)d^3\ell=\nabla P_L^\varep(\bx).
\ea\ee
Similarly, since ${\bf \Pi}_L^\varep+{\bf \Pi}_T^\varep=P^\varep{\bf 1}$ and $P_L^\varep+P_T^\varep=P^\varep$, we can also obtain $\s {\bf \Pi}_T^\varep=\nabla P_T^\varep$. Then we have proved the claim \eqref{c8}.

Now, thanks to \eqref{c8}, the equations  \eqref{c4} and \eqref{c6} can be further simplified as follows
\be\label{c15}\ba
&\partial_t \bv_X^\varep+\s (\bv\otimes \bv_X)^\varep+\nabla P_X^\varep=0,~~X=L,T,\\
&\partial_t \bomega_X^\varep+\s (\bv\otimes \bomega_X)^\varep-\s (\bomega \otimes \bv_X)^\varep=0,~~X=L,T.
\ea\ee
Next, our task is to show $\bv_L^\varep,\bv_T^\varep$ and $\bomega_L^\varep,\bomega_T^\varep$ are divergence-free in the sense of distributions.

 Actually, for any test function $\psi(\bx)\in C_0^\infty(\mathbb{T}^3)$, it follows from a straitforward  computation that
\be\label{c16}\ba
\int_{\mathbb{T}^{3}} \psi(\bx)\s \bv_L^\varep(\bx,t)d^3\bx&=\int_{\mathbb{T}^{3}}\psi(\bx)\partial_{x_i}\int_{\mathbb{T}^{3}} \varphi^\varep(\ell)\bn_i\bn_j\bv(\bx+\bl)d^3\bl d^3\bx\\
&=\int_{\mathbb{T}^{3}}\int_{\mathbb{T}^{3}}\psi(\bx) \varphi^\varep(\ell)\bn_i\bn_j\partial_{x_i}v_j(\bx+\bl)d^3\bl d^3\bx\\
&=-\int_{\mathbb{T}^{3}}\int_{\mathbb{T}^{3}}\partial_{\ell_i}\B[ \varphi^\varep(\ell)\bn_i\bn_j\B]\psi(\bx)v_j(\bx+\bl)d^3\bx d^3\bl\\ 
&=\int_{\mathbb{T}^{3}} \varphi^\varep_L(\ell)\s \big(\psi*\bv\big)(\bl)d^3\bl=0,
\ea\ee
where in the last equality we have used the fact that $\bv$ is divergence-free in the sense of ditributions. Thus, we have shown that  $\s \bv_L^\varep=0$ in the sense of distributions. Moreover, according to the fact $\bv_L^\varep+\bv_T^\varep=\bv^\varep$,  it follows that $\s  \bv_{T}^\varep=0$. Using the same argument as above, we can easily carry out $\s \bomega_L^\varep=\s \bomega _T^\varep=0$ in the sense of distributions.

Then we are going to establish the  local longitudinal K\'arm\'arth-Howarth equations for the helicity of  Euler equations. Multiplying $\eqref{c15}_1,\eqref{NS}_2, \eqref{c15}_2,\eqref{NS}_1$ with $X=L$ by $\bomega,\bv_L^\varep,\bv,\bomega_L^\varep$, respectively, we have
\begin{align}
\label{c17}
&\bomega\cdot\partial_{t} \bv_L^\varep+\bomega\cdot \s (\bv\otimes \bv_L)^\varep+\bomega\cdot \nabla P_L^\varep =0,
\\
\label{c20}
&\bv_L^\varep\cdot \partial_t \bomega+\bv_L^\varep\cdot \s (\bv \otimes \bomega)-\bv_L^\varep\cdot\s (\bomega\otimes \bv) =0,\\
\label{c18}
&\bv\cdot \partial_t \bomega^\varep_L+\bv\cdot \s (\bv \otimes \bomega_L)^\varep-\bv\cdot\s (\bomega\otimes \bv_L)^\varep =0,
\\
\label{c19}
&\bomega^\varep_L\cdot\partial_{t} \bv+\bomega^\varep_L\cdot \s (\bv\otimes \bv)+\bomega^\varep_L\cdot \nabla P =0.
\end{align}
On account of the Leibniz formula and divergence-free condition, we infer that
$$\ba
&\bomega\cdot \s (\bv\otimes \bv_L)^\varep+ \bv^\varep_L\cdot \s (\bv\otimes \bomega)=\s[\bv(\bomega\cdot\bv_L^\varep)]+\omega_j\cdot \partial_i\B[ (v_i v_{L_j})^\varep- (v_i v_{L_j}^\varep)\B],\\
&-\bv_L^\varep\cdot \s (\bomega \otimes \bv)-\bv\cdot \s (\bomega \otimes \bv_L)^\varep=-\s[\bomega(\bv\cdot\bv_L^\varep)]-v_j\partial_i\B[(\omega_i v_{L_j})^\varep-\omega_i v_{L_j}^\varep\B],\\
&\bv\cdot\s (\bv\otimes \bomega_L )^\varep+\bomega_L^\varep\cdot\s (\bv\otimes \bv)=\s[\bv (\bv\cdot\bomega_L^\varep)]+v_j\partial_i\B[(v_i\omega_{L_j})^\varep-(v_i \omega_{L_j}^\varep)\B].
\ea$$
With this in hand, after a few computations, we get
\be\label{c21}\ba
&\partial_t(\bomega\cdot \bv_L^\varep+\bv\cdot \bomega_L^\varep)+\s (\bomega P_L^\varep+P\bomega_L^\varep) +\s \B[\bv(\bomega\cdot\bv_L^\varep)-\bomega(\bv\cdot\bv_L^\varep)+\bv(\bv\cdot\bomega_L^\varep)\B]\\
=&\omega_j\partial_i\B[(v_iv_{L_j}^\varep)-(v_iv_{L_j})^\varep\B]+v_j\partial_i
\B[(v_i\omega_{L_j}^\varep)-(v_i\omega_{L_j})^\varep\B]+v_j\partial_i\B[(\omega_iv_{L_j})^\varep-(\omega_i v_{L_j}^\varep)\B].
\ea\ee
To process further, we notice that
\be\label{c22}\ba
\delta \bv_L(\ell)&=(\bn\otimes \bn)\delta \bv=\bn_i(\bn_j\delta v_j)=(\delta v_{L_1},\delta v_{L_2},\delta v_{L_3}),~\text{with}~\delta v_{L_i}=\bn_i(\bn_j \delta v_j).
\ea\ee
And for any vectors $\bm{E}=(e_1,e_2,e_3)$ and $\bm{F}=(f_1,f_2,f_3)$, it is easy to verify that
\be\label{c23}\ba
 \delta \bm{E}_{L}\cdot \delta \bm{F}_{L}=\bn_i(\bn_j \delta e_j)\cdot\bn_i(\bn_k \delta f_k)=\bn_j \bn_k\delta e_j \delta f_k=\bn_j \bn_i\delta e_j \delta f_i= \bn_i\bn_j\delta e_i \delta f_j
\ea\ee
 and
\be\label{c24}\ba
\delta\bm{E}_{T}\cdot \delta \bm{F}_{T}&=\delta {e}_{T_i}\cdot \delta {f}_{T_i}\\
&=(\delta^{ij}-\bn_i\bn_j)\delta e_j\cdot (\delta^{ik}-\bn_i\bn_k)\delta f_k\\ 
&=(\delta^{ij}-\bn_i\bn_j)\delta e_i\delta f_j.\\
\ea\ee
Combining \eqref{c23} and \eqref{c24},   we can obtain
\be\label{c25}\ba
&\int_{\mathbb{T}^{3}} \nabla \varphi^\varep(\ell)\cdot \delta \bomega (\bl)[\delta \bv_L(\bl)]^2+\f{2}{|\bl|}\varphi^\varep(\ell)\bn\cdot \delta \bomega (\bl)[\delta \bv_T(\bl)]^2 d^3\bl\\
 =&\int_{\mathbb{T}^{3}} \f{\partial \varphi^\varep}{\partial_{\ell_k}}\cdot \delta \omega_k\bn_i\bn_j\delta v_i \delta v_j+\varphi^\varep\bn_k\delta \omega_k (\f{\partial \bn_i}{\partial\ell_j}+\f{\partial \bn_j}{\partial\ell_i})\delta v_i\delta v_j d^3\bl, \ea\ee
where we have used \eqref{c11}.

Then it follows from the Leibniz formula  that
\be\ba\label{2.26}
	&\int_{\mathbb{T}^{3}} \nabla \varphi^\varep(\ell)\cdot \delta \bomega (\bl)[\delta \bv_L(\bl)]^2+\f{2}{|\bl|}\varphi^\varep(\ell)\bn\cdot \delta \bomega (\bl)[\delta \bv_T(\bl)]^2 d^3\bl\\
 =&\int_{\mathbb{T}^{3}} \f{\partial}{\partial_{\ell_k}}\B(\varphi^\varep \bn_i\bn_j\B)\delta \omega_k\delta v_i\delta v_j -\varphi^\varep\B[\f{\partial}{\partial_{\ell_k}}(\bn_i\bn_j)-(\f{\partial \bn_i}{\partial\ell_j}+\f{\partial \bn_j}{\partial\ell_i})\bn_k\B]\delta \omega_k \delta v_i\delta v_j d^3\bl. \ea\ee
Before going any further,   assume we have proved that, for any vector fields ${\bf A}=(A_1,A_2,A_3), {\bf B}=(B_1,B_2,B_3)$ and ${\bf C}=(C_1,C_2,C_3)$, there holds
 \be\label{2.27}\ba
&\B[\f{\partial}{\partial_{\ell_k}}(\bn_i\bn_j)-(\f{\partial \bn_i}{\partial\ell_j}+\f{\partial \bn_j}{\partial\ell_i})\bn_k\B]A_kB_iC_j\\
=&\f{1}{|\bl|}\bn\cdot\B[{\bf C}({\bf A}\cdot {\bf B})+{\bf B}({\bf A}\cdot {\bf C})-2{\bf A}({\bf B}\cdot {\bf C})\B].
\ea\ee
As a consequent, by choosing ${\bf A}=\delta \bomega, {\bf B}={\bf C}=\delta \bv$ in \eqref{2.27}, we have
\be\label{c33}\ba
&\B[\f{\partial}{\partial_{\ell_k}}(\bn_i\bn_j)-(\f{\partial \bn_i}{\partial\ell_j}+\f{\partial \bn_j}{\partial\ell_i})\bn_k\B]\delta \omega_k\delta v_i \delta v_j\\
 =&\f{2}{|\bl|}\bn\cdot\B[\delta \bv(\delta \bomega\cdot \delta \bv)-\delta \bomega(\delta \bv\cdot \delta \bv)\B].
\ea\ee
Substituting \eqref{c33} into \eqref{2.26}, it yields that
\be\ba\label{2.29}	&\int_{\mathbb{T}^{3}} \nabla \varphi^\varep(\ell)\cdot \delta \bomega (\bl)[\delta \bv_L(\bl)]^2+\f{2}{|\bl|}\varphi^\varep(\ell)\bn\cdot \delta \bomega (\bl)[\delta \bv_T(\bl)]^2 d^3\bl\\=&\int_{\mathbb{T}^{3}} \f{\partial}{\partial_{\ell_k}}\B(\varphi^\varep \bn_i\bn_j\B)\delta \omega_k\delta v_i\delta v_j -\f{2}{|\bl|}\varphi^\varep\bn\cdot\B[\delta \bv(\delta \bomega \cdot \delta \bv)-\delta \bomega (\delta \bv\cdot \delta \bv)\B]d^3\bl,
\ea\ee
Next, for simplicity of presentation, we denote
\be\label{2.30}
\bv(\bx+\bl)=\bar{\bv}=(\bar{v}_1,\bar{v}_2,\bar{v}_3 ),  \bomega(\bx+\bl)=\bar{\bomega}=(\bar{\omega}_1,\bar{\omega}_2,\bar{\omega}_3 ).\ee
A routine computation gives rise to
\be\label{2.31}\ba
&\int_{\mathbb{T}^{3}} \f{\partial}{\partial_{\ell_k}}\B(\varphi^\varep \bn_i\bn_j\B)\delta \omega_k\delta v_i\delta v_j d^3\bl\\
=&\int_{\mathbb{T}^{3}} \f{\partial}{\partial_{\ell_k}}\B(\varphi^\varep \bn_i\bn_j\B)\B[{\omega}_k(\bx+\bl)-\omega_k(\bx)\B]\B[{v}_i(\bx+\bl)-v_i(\bx)\B]
\B[{v}_j(\bx+\bl)-v_j(\bx)\B]d^3\bl\\
 =&\int_{\mathbb{T}^{3}} \f{\partial}{\partial_{\ell_k}}\B(\varphi^\varep \bn_i\bn_j\B)\\&\times\B(\bar{\omega}_k \bar{v}_i\bar{v}_j -\bar{\omega}_k\bar{v}_i v_j-\bar{\omega}_kv_i\bar{v}_j+\bar{\omega}_kv_iv_j-\omega_k\bar{v}_i\bar{v}_j+\omega_k\bar{v}_iv_j+\omega_kv_i\bar{v}_j-\omega_kv_iv_j\B)d^3\bl.
\ea\ee
Then we will deal with  the terms on the right-hand side of \eqref{2.31} one by one. For example, according to integration by parts and \eqref{c23}, we get
\be\label{2.32}\ba
&\int_{\mathbb{T}^{3}} \f{\partial}{\partial_{\ell_k}}\B(\varphi^\varep \bn_i\bn_j\B)\bar{\omega}_k\bar{v}_i\bar{v}_{j}d^3\bl\\
=&-\int_{\mathbb{T}^{3}} \varphi^\varep \bn_i\bn_j\f{\partial}{\partial_{\ell_k}}\B(\bar{\omega}_k\bar{v}_i\bar{v}_{j}\B)d^3\bl\\
=&-\f{\partial}{\partial_{x_k}}\int_{\mathbb{T}^{3}} \varphi^\varep \bn_i\bn_j\bar{\omega}_k\bar{v}_i\bar{v}_{j}d^3\bl\\ 
%=&-\partial_k\B(\omega_k (\bv_L\cdot \bv_L)\B)^\varep\\
=&-\s\B(\bomega(\bv_L\cdot\bv_L)\B)^\varep.
\ea\ee
Performing the same procedure for the rest terms on the right-hand side of \eqref{2.31}, we have
\be\label{2.33}\ba
&-\int_{_{\mathbb{T}^{3}}} \f{\partial}{\partial_{\ell_k}}\B[\varphi^\varep \bn_i\bn_j\B]\bar{\omega}_k\bar{v}_iv_j d^3\bl=v_j\partial_k(\omega_k v_{L_j})^\varep,\\
&-\int_{\mathbb{T}^{3}} \f{\partial}{\partial_{\ell_k}}\B[\varphi^\varep \bn_i\bn_j\B]\bar{\omega}_kv_i\bar{v}_j d^3\bl=v_i\partial_k(\omega_k v_{L_i})^\varep,\\
&\int_{\mathbb{T}^{3}} \f{\partial}{\partial_{\ell_k}}\B[\varphi^\varep \bn_i\bn_j\B]\bar{\omega_k}v_iv_j d^3\bl=-(\bv_L\cdot\bv_L)\partial_k(\omega_k)^\varep=0,\\
&-\int_{\mathbb{T}^{3}} \f{\partial}{\partial_{\ell_k}}\B[\varphi^\varep \bn_i\bn_j\B]\omega_k\bar{v}_i\bar{v}_jd^3\bl=\omega_k\partial_k(\bv_L\cdot \bv_L)^\varep=\s \B[\bomega(\bv_L\cdot \bv_L)^\varep\B],\\
&\int_{\mathbb{T}^{3}} \f{\partial}{\partial_{\ell_k}}\B[\varphi^\varep \bn_i\bn_j\B]\omega_k\bar{v}_iv_jd^3\bl =
-\omega_k v_j\partial_k(v_{L_j})^\varep=-v_j\partial_k(\omega_k v_{L_j}^\varep),\\
&\int_{\mathbb{T}^{3}} \f{\partial}{\partial_{\ell_k}}\B[\varphi^\varep \bn_i\bn_j\B]\omega_kv_i \bar{v}_j d^3 \bl =-\omega_k v_i\partial_k v_{L_i}^\varep=-v_i\partial_k(\omega_k v_{L_i}^\varep),\\
&-\int_{\mathbb{T}^{3}} \f{\partial}{\partial_{\ell_k}}\B[\varphi^\varep \bn_i\bn_j\B]\omega_kv_iv_j d^3\bl =0,
\ea\ee
where the divergence-free condition of the velocity and the vorticity was used many times.
Then plugging \eqref{2.32} and \eqref{2.33} into \eqref{2.31}, we can obtain
\be\label{2.34}\ba
 &\int_{\mathbb{T}^{3}} \f{\partial}{\partial_{\ell_k}}\B[\varphi^\varep \bn_i\bn_j\B]\delta \omega_k\delta v_i\delta v_j d^3\bl\\
 =&-\s \B[\big(\bomega(\bv_L\cdot\bv_L)\big)^\varep-\bomega(\bv_L\cdot\bv_L)^\varep\B]+2v_j\partial_k\B[(\omega_k v_{L_j})^\varep-(\omega_k v_{L_j}^\varep)\B].
\ea\ee
Furthermore, by inserting \eqref{2.34} into \eqref{2.29}, we observe that
\be\label{2.35}\ba
&\f12\int_{\mathbb{T}^{3}} \nabla \varphi^\varep(\ell)\cdot \delta \bomega [\delta \bv_L]^2+\f{2}{|\bl|}\varphi^\varep(\ell)\bn\cdot \B[\delta \bomega [\delta \bv_T]^2 +\delta \bv(\delta \bomega \cdot \delta \bv)-\delta \bomega (\delta \bv\cdot \delta \bv)\B]d^3\bl\\
&+\f12\s \B[\big(\bomega(\bv_L\cdot\bv_L)\big)^\varep-\bomega(\bv_L\cdot\bv_L)^\varep\B]
\\=& v_j\partial_k\B[(\omega_k v_{L_j})^\varep-(\omega_k v_{L_j}^\varep)\B].
\ea\ee
Along the same lines as \eqref{2.26}, we find
\be\label{2.36}\ba
&\int_{\mathbb{T}^{3}} \nabla \varphi^\varep(\ell)\cdot \delta \bv(\bl)(\delta \bv_L \cdot \delta \bomega_L)+\f{2}{|\bl|}\varphi^\varep(\ell)\bn\cdot \delta \bv(\bl)(\delta \bv_T\cdot \delta\bomega_T) d^3 \bl\\
 =&\int_{\mathbb{T}^{3}} \f{\partial }{\partial_{\ell_k}}\B(\varphi^\varep\bn_i\bn_j\B)\delta v_k\delta v_i \delta \omega_j-\varphi^\varep\B[\f{\partial}{\partial_{\ell_k}}(\bn_i\bn_j)-(\f{\partial \bn_i}{\partial\ell_j}+\f{\partial \bn_j}{\partial\ell_i})\bn_k\B]\delta v_k\delta v_i \delta \omega_j d^3\bl
\ea\ee
From \eqref{2.27}, we reformulate the above equation  as
\be\ba\label{2.37}
&\int_{\mathbb{T}^{3}} \nabla \varphi^\varep(\ell)\cdot \delta \bv(\bl)(\delta \bv_L \cdot \delta \bomega_L)+\f{2}{|\bl|}\varphi^\varep(\ell)\bn\cdot \delta \bv(\bl)(\delta \bv_T\cdot \delta\bomega_T) d^3 \bl\\
=&\int_{\mathbb{T}^{3}} \f{\partial }{\partial_{\ell_k}}\B(\varphi^\varep\bn_i\bn_j\B)\delta v_k\delta v_i \delta \omega_j+\f{1}{|\bl|}\varphi^\varep\bn\cdot\B[\delta \bv(\delta \bomega \cdot \delta \bv)-\delta \bomega (\delta \bv\cdot \delta \bv)\B]d^3\bl.
\ea\ee
Repeating the deduction process of \eqref{2.34}, we arrive at
\be\label{2.38}\ba
 &\int_{\mathbb{T}^{3}} \f{\partial }{\partial_{\ell_k}}\B(\varphi^\varep\bn_i\bn_j\B)\delta v_k\delta v_i \delta \omega_j d^3 \bl\\
=&-\s \B[\big(\bv (\bv_L\cdot \bomega _L)\big)^\varep-\bv (\bv_L\cdot \bomega _L)^\varep\B]+\omega_j\partial_k\B[(v_k v_{L_j})^\varep-(v_k v_{L_j}^\varep)\B]+v_i\partial_k\B[(v_k \omega_{L_i})^\varep-(v_k \omega_{L_i}^\varep)\B].
\ea\ee
Then, it follows from \eqref{2.37} and \eqref{2.38} that
 \be\label{2.39}\ba
&\int_{\mathbb{T}^{3}} \nabla \varphi^\varep(\ell)\cdot \delta \bv(\delta \bv_L \cdot \delta \bomega_L)+\f{2}{|\bl|}\varphi^\varep(\ell)\bn\cdot  \B[\delta\bv(\delta \bv_T\cdot \delta\bomega_T)-\f12\big(\delta \bv(\delta \bomega \cdot \delta \bv)-\delta \bomega (\delta \bv\cdot \delta \bv)\big)\B] d^3 \bl\\&+\s \B[\big(\bv (\bv_L\cdot \bomega _L)\big)^\varep-\bv (\bv_L\cdot \bomega _L)^\varep\B]\\
=&\omega_j\partial_k\B[(v_k v_{L_j})^\varep-(v_k v_{L_j}^\varep)\B]+v_i\partial_k\B[(v_k \omega_{L_i})^\varep-(v_k \omega_{L_i}^\varep)\B].
\ea\ee
Finally, substituting \eqref{2.35} and \eqref{2.39} into \eqref{c21}, we conclude that
\be\ba\label{2.40v1}
&\partial_t(\bomega\cdot \bv_L^\varep+\bv\cdot \bomega_L^\varep)+\s (\bomega P_L^\varep+P\bomega_L^\varep) +\s \B[\bv(\bomega\cdot\bv_L^\varep)-\bomega(\bv\cdot\bv_L^\varep)+\bv(\bv\cdot\bomega_L^\varep)\B]\\
&+\s \B[\big(\bv (\bv_L\cdot \bomega _L)\big)^\varep-\bv (\bv_L\cdot \bomega _L)^\varep\B]
-\f12\s \B[\big(\bomega(\bv_L\cdot\bv_L)\big)^\varep-\bomega(\bv_L\cdot\bv_L)^\varep\B]
\\
=&-\f43 D_{HL}^\varep(\bv,\bomega),
\ea\ee
where
\be\label{c41}\ba
&D_{HL}^\varep(\bv,\bomega)\\
=&\f34\int \nabla \varphi^\varep(\ell)\cdot \delta \bv(\delta \bv_L \cdot \delta \bomega_L)+\f{2}{|\bl|}\varphi^\varep(\ell)\bn\cdot \B[\delta \bv(\delta \bv_T\cdot \delta\bomega_T)+\delta \bv\times (\delta \bomega \times \delta \bv)\B] d^3 \bl\\
&-\f38 \int_{\mathbb{T}^{3}} \nabla \varphi^\varep(\ell)\cdot \delta \bomega [\delta \bv_L]^2+\f{2}{|\bl|}\varphi^\varep(\ell)\bn\cdot \delta \bomega [\delta \bv_T]^2d^3\bl,\\
\ea\ee
and in the last equality we have used the following vector identity
\be\label{c41-1}\delta \bv\times (\delta \bomega \times \delta \bv)=\delta \bomega (\delta \bv\cdot \delta \bv) -\delta \bv(\delta \bomega \cdot \delta \bv).\ee
Then we have proved the local longitudinal K\'arm\'arth-Howarth equations for helicity. Next, we are in a position to get the  transverse K\'arm\'arth-Howarth equations for helicity. Actually, arguing in the same manner as in the derivation of \eqref{c21}, we have
\be\label{2.41}\ba
&\partial_t(\bomega\cdot \bv_T^\varep+\bv\cdot \bomega_T^\varep)+\s (\bomega P_T^\varep +P\bomega_T^\varep) +\s\B[\bv (\bomega\cdot \bv_T^\varep)+\bv(\bv\cdot\bomega_T^\varep)-\bomega(\bv\cdot\bv_T^\varep)\B]\\
=&\omega_j\partial_i\B[(v_i v_{T_j}^\varep)-(v_i v_{T_j})^\varep\B]+v_j\partial_i\B[(v_i\omega_{T_j}^\varep)-(v_i\omega_{T_j})^\varep\B]+v_j\partial_i\B[(\omega_i v_{T_j})^\varep-(\omega_i v_{T_j}^\varep)\B].
\ea\ee
In view of \eqref{c11}, \eqref{c24} and the Leibniz formula, we see that
\be\label{2.43}\ba
&\int_{\mathbb{T}^{3}} \nabla \varphi^\varep(\ell)\cdot \delta \bomega[\delta \bv_T]^2-\f{2}{|\bl|}\varphi^\varep\bn\cdot \delta \bomega [\delta \bv_T]^2 d^3\bl\\
 =&\int_{\mathbb{T}^{3}} \partial_{\ell_k}\B[\varphi^\varep(\delta^{ij}-\bn_i\bn_j)\B]\delta \omega_k\delta v_i \delta v_j+\varphi^\varep\B[\partial_{\ell_k}(\bn_i\bn_j)-(\f{\partial \bn_i}{\partial_{\ell_j}}+\f{\partial \bn_j}{\partial_{\ell_i}})\bn_k\B]\delta \omega_k\delta v_i \delta v_j d^3\bl,\\
\ea\ee
which together with \eqref{c33} and \eqref{c41-1} yields
\be\ba
&\int_{\mathbb{T}^{3}} \nabla \varphi^\varep(\ell)\cdot \delta \bomega[\delta \bv_T]^2-\f{2}{|\bl|}\varphi^\varep\bn\cdot \delta \bomega [\delta \bv_T]^2 d^3\bl \\
=&\int_{\mathbb{T}^{3}} \partial_{\ell_k}\B[\varphi^\varep(\delta^{ij}-\bn_i\bn_j)\B]\delta \omega_k\delta v_i \delta v_j-\f{2}{|\bl|}\varphi^\varep\bn\cdot \big[ \delta \bv\times(\delta \bomega\times \delta \bv)\big] d^3\bl
\ea\ee
Likewise, combining \eqref{2.27} and \eqref{c41-1}, we know that
 \be\label{2.44}\ba
&\int_{\mathbb{T}^{3}} \nabla \varphi^\varep(\ell)\cdot \delta \bm{v}(\delta \bm{\omega}_T\cdot \delta \bm{v}_T)-\f{2}{|\bl|}\varphi^\varep\bn\cdot \delta\bm{v} (\delta\bm{\omega}_T\cdot \delta\bm{v}_T) d^3\bl\\
=&\int_{\mathbb{T}^{3}} \partial_{\ell_k}\varphi^\varep\cdot\delta v_k (\delta^{ij}-\bn_i\bn_j)\delta \omega_i\delta v_j -\f{2}{|\bl|}\varphi^\varep\bn_k \cdot\delta v_k(\delta^{ij}-\bn_i\bn_j)\delta \omega_i \delta v_jd^3\bl\\
=&\int_{\mathbb{T}^{3}} \partial_{\ell_k}\B[\varphi^\varep(\delta^{ij}-\bn_i\bn_j)\B]\delta v_k\delta \omega_i \delta v_j+\varphi^\varep\B[\partial_{\ell_k}(\bn_i\bn_j)-(\f{\partial \bn_i}{\partial_{\ell_j}}+\f{\partial \bn_j}{\partial_{\ell_i}})\bn_k\B]\delta v_k\delta \omega_i \delta v_j d^3\bl\\
=&\int_{\mathbb{T}^{3}} \partial_{\ell_k}\B[\varphi^\varep(\delta^{ij}-\bn_i\bn_j)\B]\delta v_k\delta \omega_i \delta v_j+\f{1}{|\bl|}\varphi^\varep\bn\cdot \big[ \delta \bv\times(\delta \bomega\times \delta \bv)\big]  d^3\bl
\ea\ee
To control the first term on the right-hand side of \eqref{2.44}, it follows from a straightforward calculation that
$$\ba
&\int_{\mathbb{T}^{3}} \partial_{\ell_k}\B[\varphi^\varep(\delta^{ij}-\bn_i\bn_j)\B]\delta v_k\delta \omega_i \delta v_j d^3\bl\\
=&\int_{\mathbb{T}^{3}} \partial_{\ell_k}\B[\varphi^\varep(\delta^{ij}-\bn_i\bn_j)\B]\\&\times
[\overline{v}_k\overline{\omega}_{i}\overline{v}_{j}
-\overline{v}_k\overline{\omega}_{i} {v}_{j}
-\overline{v}_k {\omega}_{i}\overline{v}_{j}
+\overline{v}_k {\omega}_{i} {v}_{j}-{v}_k\overline{\omega}_{i}\overline{v}_{j}
+ {v}_k\overline{\omega}_{i} {v}_{j}
+ {v}_k {\omega}_{i}\overline{v}_{j}
- {v}_k {\omega}_{i} {v}_{j} ]d^3\bl,
\ea$$
where we have used the notation \eqref{2.30}.
For example, by virtue of  integration by parts and \eqref{c24}, we obtain
\be\label{c35}\ba
 \int_{\mathbb{T}^{3}} \f{\partial}{\partial_{\ell_k}} \B[\varphi^\varep(\delta^{ij}-\bn_i\bn_j)\B]\bar{v}_k\bar{\omega}_i\bar{v}_{j}
d^3\bl=&-\int_{\mathbb{T}^{3}} \B[\varphi^\varep(\delta^{ij}-\bn_i\bn_j)\B]\f{\partial}{\partial_{\ell_k}}\B(\bar{v}_k\bar{\omega}_i\bar{v}_{j}\B)d^3\bl\\
=&-\f{\partial}{\partial_{x_k}}\int_{\mathbb{T}^{3}} \B[\varphi^\varep(\delta^{ij}-\bn_i\bn_j)\B]\bar{v}_k\bar{\omega}_i\bar{v}_{j}d^3\bl\\
=&-\s\B[\bv (\bomega_T\cdot \bv_T)\B]^\varep.
\ea\ee
Similar to the above derivation, it follows that
\be\label{2.47}\ba
&-\int_{_{\mathbb{T}^{3}}} \f{\partial}{\partial_{\ell_k}}\B[\varphi^\varep(\delta^{ij}-\bn_i\bn_j)\B]\bar{v}_k\bar{\omega}_i v_{j}d^3\bl=v_j\partial_k(v_k \omega_{T_j})^\varep,\\
&-\int_{\mathbb{T}^{3}} \f{\partial}{\partial_{\ell_k}}\B[\varphi^\varep(\delta^{ij}-\bn_i\bn_j)\B]\bar{v}_k\omega_i\bar{v}_j d^3\bl=\omega_i\partial_k(v_k v_{T_i})^\varep,\\
&\int_{\mathbb{T}^{3}} \f{\partial}{\partial_{\ell_k}}\B[\varphi^\varep(\delta^{ij}-\bn_i\bn_j)\B]\bar{v}_k\omega_iv_j d^3\bl=-\bomega_T\cdot \bv_T\int_{\mathbb{T}^{3}} \varphi^\varep \f{\partial}{\partial_{\ell_k}}\bar{v}_k d^3\bl=0,\\
&-\int_{\mathbb{T}^{3}} \f{\partial}{\partial_{\ell_k}}\B[\varphi^\varep(\delta^{ij}-\bn_i\bn_j)\B]v_k\bar{\omega}_i\bar{v}_jd^3\bl
=v_k\partial_k(\bomega_T\cdot \bv_T)^\varep
=\s [\bv(\bomega_T\cdot \bv_T)^\varep],\\
&\int_{\mathbb{T}^{3}} \f{\partial}{\partial_{\ell_k}}
\B[\varphi^\varep(\delta^{ij}-\bn_i\bn_j)\B]v_k\bar{\omega}_iv_jd^3\bl =
-v_k v_j\partial_k(\omega_{T_j})^\varep=-v_j\partial_k(v_k \omega_{T_j}^\varep),\\
&\int_{\mathbb{T}^{3}} \f{\partial}{\partial_{\ell_k}}\B[\varphi^\varep(\delta^{ij}-\bn_i\bn_j)\B]v_k\omega_i\bar{v}_jd^3 \bl =-v_k \omega_i\partial_k v_{T_i}^\varep=-\omega_i\partial_k(v_k v_{T_i}^\varep),\\
&-\int_{\mathbb{T}^{3}} \f{\partial}{\partial_{\ell_k}}\B[\varphi^\varep(\delta^{ij}-\bn_i\bn_j)\B]v_k\omega_iv_j d^3\bl =0.
\ea\ee
Consequently, combining \eqref{2.44} through \eqref{2.47}, we get
\be\label{2.48}\ba
&\int_{\mathbb{T}^{3}} \nabla \varphi^\varep(\ell)\cdot \delta \bm{v}(\delta \bm{\omega}_T\cdot \delta \bm{v}_T)-\f{2}{|\bl|}\varphi^\varep\bn\cdot \delta\bm{v} (\delta\bm{\omega}_T\cdot \delta\bm{v}_T)-\f{1}{|\bl|}\varphi^\varep\bn\cdot \big[ \delta \bv\times(\delta \bomega\times \delta \bv)\big] d^3\bl \\
=&\int_{\mathbb{T}^{3}} \partial_{\ell_k}\B[\varphi^\varep(\delta^{ij}-\bn_i\bn_j)\B]\delta v_k\delta \omega_i \delta v_jd^3\bl\\
=&-\s\B[\big(\bv (\bomega_T\cdot \bv_T)\big)^\varep-\bv(\bomega_T\cdot \bv_T)^\varep]+v_j\partial_k\B[(v_k \omega_{T_j})^\varep-(v_k \omega_{T_j}^\varep)\B]+\omega_i\partial_k\B[(v_k v_{T_i})^\varep-(v_k v_{T_i}^\varep)\B],
\ea\ee
which implies
\be\label{2.49}\ba
&\int_{\mathbb{T}^{3}} \nabla \varphi^\varep(\ell)\cdot \delta \bm{v}(\delta \bm{\omega}_T\cdot \delta \bm{v}_T)-\f{2}{|\bl|}\varphi^\varep\bn\cdot \delta\bm{v} (\delta\bm{\omega}_T\cdot \delta\bm{v}_T)-\f{1}{|\bl|}\varphi^\varep\bn\cdot \big[ \delta \bv\times(\delta \bomega\times \delta \bv)\big] d^3\bl \\&+\s\B[\big(\bv (\bomega_T\cdot \bv_T)\big)^\varep-\bv(\bomega_T\cdot \bv_T)^\varep\B]\\
=&v_j\partial_k\B[(v_k \omega_{T_j})^\varep-(v_k \omega_{T_j}^\varep)\B]+\omega_i\partial_k\B[(v_k v_{T_i})^\varep-(v_k v_{T_i}^\varep)\B].
\ea\ee
On the other hand, by a suitable modification to the  derivation of \eqref{2.48}, we infer that
\be\label{2.50}\ba
&\int_{\mathbb{T}^{3}} \nabla \varphi^\varep(\ell)\cdot \delta \bm{\omega}[\delta \bm{v}_T]^2-\f{2}{|\bl|}\varphi^\varep\bn\cdot \delta \bm{\omega} [\delta \bm{v}_T]^2 +\f{2\varphi^\varep}{|\bl|}\bn\cdot\big[ \delta \bv\times(\delta \bomega\times \delta \bv)\big]d^3\bl\\
=&\int_{\mathbb{T}^{3}} \partial_{\ell_k}\B[\varphi^\varep(\delta^{ij}-\bn_i\bn_j)\B]\delta \omega_k\delta v_i \delta v_j d^3\bl\\
  =&-\s \B[\big(\bm{\omega}(\bm{v}_T\cdot\bm{v}_T)\big)^\varep-\bm{\omega}(\bm{v}_T\cdot\bm{v}_T)^\varep\B]+2v_j\partial_k\B[(\omega_k v_{T_j})^\varep-(\omega_k v_{T_j}^\varep)\B],
\ea\ee
which means
\be\label{2.51}\ba
&\f12\int_{\mathbb{T}^{3}} \nabla \varphi^\varep(\ell)\cdot \delta \bm{\omega}[\delta \bm{v}_T]^2-\f{2}{|\bl|}\varphi^\varep\bn\cdot \delta \bm{\omega} [\delta \bm{v}_T]^2 +\f{2}{|\bl|}\varphi^\varep\bn\cdot\big[ \delta \bv\times(\delta \bomega\times \delta \bv)\big]d^3 \bl\\ &+\f12\s \B[\big(\bm{\omega}(\bm{v}_T\cdot\bm{v}_T)\big)^\varep-\bm{\omega}(\bm{v}_T\cdot\bm{v}_T)^\varep\B]
\\
=&v_j\partial_k\B[(\omega_k v_{T_j})^\varep-(\omega_k v_{T_j}^\varep)\B].
\ea\ee
Thus, we plug \eqref{2.49} and \eqref{2.51} into  \eqref{2.41} to deduce that
\be\label{2.52}\ba
&\partial_t(\bomega\cdot \bv_T^\varep+\bv\cdot \bomega_T^\varep)+\s (\bomega P_T^\varep +P\bomega_T^\varep) +\s\B[\bv (\bomega\cdot \bv_T^\varep)+\bv(\bv\cdot\bomega_T^\varep)-\bomega(\bv\cdot\bv_T^\varep)\B]\\
&+\s\B[\big(\bv (\bomega_T\cdot \bv_T)\big)^\varep-\bv(\bomega_T\cdot \bv_T)^\varep\B] -\f12\s \B[\big(\bm{\omega}(\bm{v}_T\cdot\bm{v}_T)\big)^\varep-\bm{\omega}(\bm{v}_T\cdot\bm{v}_T)^\varep\B]
\\ 
=&-\f83D_{HT}^\varep(\bv,\bomega),
\ea\ee
where
\be\label{2.53}\ba
&D_{HT}^\varep(\bv,\bomega)\\
=&\f38\int_{\mathbb{T}^{3}}\nabla \varphi^\varep(\ell)\cdot \delta\bv(\delta \bv_T\cdot \delta \bomega_T)-\f{2}{|\bl|}\varphi^\varep(\ell)\bn\cdot\B[\delta \bv(\delta \bv_T\cdot \delta \bomega_T)+\delta \bv\times (\delta \bomega \times \delta \bv)\B]d^3\bl\\
&-\f{3}{16}\int_{\mathbb{T}^{3}} \nabla \varphi^\varep(\ell)\cdot \delta \bomega[\delta \bv_T]^2-\f{2}{|\bl|}\varphi^\varep(\ell)\bn\cdot\delta \bomega[\delta \bv_T]^2 d^3\bl.
\ea\ee
Then the validity of transverse K\'arm\'arth-Howarth equations for helicity has been proved.

\underline{{\bf Step 2}}: The next objective is passing to the limit of $\varep$ on the left hand side of \eqref{2.40v1} and \eqref{2.52}. Firstly, due to the definition of $\varphi^\varep$, we have the following elementary fact
\be\label{2.54}
\int_{\mathbb{T}^{3}} \varphi^\varep(\ell)(\bn\otimes \bn)d^3 \bl=\f13 {\bf 1}.
\ee
Together with  the definition of $\bv_L^\varep$ and using the Minkowski inequality, we can obtain
\be\label{2.55}\ba
\B\|\bv_L^\varep(\bx,t)-\f13\bv(\bx,t)\B\|_{L^q(\mathbb{T}^3)}&=
\B\|\int_{\mathbb{T}^{3}} \varphi^\varep(\ell)(\bn\otimes \bn)\cdot\big[\bv(\bx+\bl,t)-\bv(\bx,t)\big]d^3\bl\B\|_{L^q(\mathbb{T}^3)}\\
&\leq \int_{\mathbb{T}^{3}} \varphi^\varep(\ell)\|\bv(\bx+\bl,t)-\bv(\bx,t)\|_{L^q(\mathbb{T}^3)}d^3\bl,
\ea\ee
which implies $\|\bv_L^\varep(\bx,t)-\f13\bv(\bx,t)\|_{L^q(\mathbb{T}^3)}\to 0$ as $\varep\to 0$ according to the fact $\bv \in L^q(\mathbb{T}^3)$. Since $\bv_T^\varep=\bv^\varep-\bv_L^\varep$, we can deduce that $\|\bv_T^\varep(\bx,t)-\f23\bv(\bx,t)\|_{L^q(\mathbb{T}^3)}\to 0$ as $\varep\to 0$. In the same manner, since $\bomega \in L^n(\mathbb{T}^3)$, we can get that $\|\bomega_L^\varep(\bx,t)-\f13\bomega(\bx,t)\|_{L^n(\mathbb{T}^3)}\to 0$  and $\|\bomega_T^\varep(\bx,t)-\f23\bomega(\bx,t)\|_{L^n(\mathbb{T}^3)}\to 0$. Moreover, in light of Calder\'on-Zygmund theorem  and $\bv\in L^q(\mathbb{T}^3)$, we know that $P\in L^{\f{q}{2}}(\mathbb{T}^3)$. Consequently, we also deduce that  $\|P_L^\varep(\bx,t)-\f13P(\bx,t)\|_{L^{\f{q}{2}}(\mathbb{T}^3)}\to 0$ and $\|P_T^\varep(\bx,t)-\f23P(\bx,t)\|_{L^{\f{q}{2}}(\mathbb{T}^3)}\to 0$ as $\varep\to 0$. Then using the standard properties of mollification, we have $\B[\big(\bv(\bomega_L\cdot \bv_L)\big)^\varep-\bv(\bomega_L\cdot \bv_L)^\varep\B]\to 0$ and $\B[\bomega (\bv_L\cdot \bv_L)^\varep-\big(\bomega(\bv_L\cdot\bv_L)\big)^\varep\B]\to 0$, provided $\f1n +\f2q=1$. Hence, letting $\varep\to 0$, the left hand side of \eqref{2.40v1}   converges to
\be\label{2.56}\ba
&\f43\partial_t(\f12\bomega\cdot \bv)+\f43\s (\f12\bomega P)+\f43\s \B[\f12\bv(\bomega\cdot\bv)-\f14\bomega(\bv\cdot\bv)\B]=:-\f43D_H(\bv,\bomega)\\
\ea\ee
in the sense of distributions. Similarly, we can also prove the left hand side of  \eqref{2.52} converges to $-\f83D_H(\bv,\bomega)$, as $\varep\to 0$, when $X=T$. In conclusion, we have
\be\label{2.57-1}D^\varep_{HX}(\bv,\bomega)\to D_{H}(\bv,\bomega),\ee
for both $X=L,T$, as $\varep\to 0$, in the sense of distributions.

\underline{{\bf Step 3}}: Next, we have to show the 4/5 law for the helictiy in the incompressible fluids based on the previous the local longitudinal and transverse K\'arm\'arth-Howarth equations for the helicty of  Euler equations \eqref{NS} we have just proved.  To this end, we let
$$\ba
&\overline{S}_{HL}(\bm{v}, \bomega,\lambda)=\f{1}{\lambda}\int_{\partial B } \bn \cdot \B[  \delta \bv(\lambda\bl)(\delta \bv_L(\lambda\bl)\cdot \delta \bomega_L(\lambda\bl))-\f12\delta \bomega(\lambda\bl)[\delta \bv_L(\lambda\bl)]^2\B] \f{d\sigma(\bl) }{4\pi},\\&\overline{S}_{HT}(\bm{v}, \bomega,\lambda)=\f{1}{\lambda}\int_{\partial B } \bn \cdot \B[  \delta \bv(\lambda\bl)(\delta \bv_T(\lambda\bl)\cdot \delta \bomega_T(\lambda\bl))-\f12\delta \bomega(\lambda\bl)[\delta \bv_T(\lambda\bl)]^2\B] \f{d\sigma(\bl) }{4\pi},\\&\overline{S}_{H} (\bm{v}, \bomega,\lambda)=\f{1}{\lambda}\int_{\partial B } \bn \cdot\B[\delta\bv(\lambda\bl)\times (\delta \bomega(\lambda\bl) \times \delta \bv(\lambda\bl))\B] \f{d\sigma(\bl) }{4\pi}.
\ea$$
By  taking a direct computation and using the radially symmetric of the test function $\varphi^\varep(\ell)$, we  conclude by coarea formula and the change of variable that
\be\label{2.57}\ba
&D_{H}(\bv,\bomega)\\=	&\lim_{\varepsilon\rightarrow0}D^\varepsilon_{HT}(\bv,\bomega )\\
 =&\lim_{\varepsilon\rightarrow0}\f38\int_{\mathbb{T}^3} \B[\nabla \varphi^\varep(\ell)-\f{2}{|\bl|}\varphi^\varep \bn\B]\B[  \delta \bv(\delta \bv_T\cdot \delta \bomega_T)-\f12\delta \bomega[\delta \bv_T]^2\B]d^3\bl \\&-\lim_{\varepsilon\rightarrow0}\f34\int_{\mathbb{T}^3} \f{1}{|\bl|}\varphi^\varep \bn\cdot \delta \bv\times (\delta \bomega \times \delta \bv)d^3\bl\\
=&\lim_{\varepsilon\rightarrow0}\f38\B(\int_0^\infty r^3\varphi^{'}(r)-2r^2\varphi(r)dr4\pi \bar{S}_{HT}(\bv,\bomega,\varep\ell)\B)-\lim_{\varepsilon\rightarrow0}\f34\B(\int_0^\infty r^2\varphi(r)dr 4\pi\bar{S}_{H}(\bv,\bomega,\varep\ell)\B)\\
=&-\f{15}{8}\B(\bar{S}_{HT}(\bv,\bomega)+\f25\bar{S}_{H}(\bv,\bomega)\B)\\
=&-\f{15}{8} S_{HT}(\bv,\bomega),
\ea\ee
which leads to
$$S_{HT}(\bv,\bomega)= -\f{8}{15} D_{H}(\bv,\bomega),$$
where
\be\ba
S_{HT}(\bv,\bomega)=&\lim_{\lambda\to 0}\f{1}{\lambda}\int_{\partial B } \bn \cdot \B[  \delta \bv(\lambda\bl)(\delta \bv_T(\lambda\bl)\cdot \delta \bomega_T(\lambda\bl))-\f12\delta \bomega(\lambda\bl)[\delta \bv_T(\lambda\bl)]^2\B] \f{d\sigma(\bl) }{4\pi}\\
&+\f25\lim_{\lambda \to0} \f{1}{\lambda}\int_{\partial B } \bn \cdot\B[\delta\bv(\lambda\bl)\times (\delta \bomega(\lambda\bl) \times \delta \bv(\lambda\bl))\B] \f{d\sigma(\bl) }{4\pi}.
\ea\ee
Following exactly the lines of reasoning which led to \eqref{2.57}, we find that
\be\label{2.58}\ba
&D_{H}(\bv,\bomega)\\=&\lim\limits_{\varep\to0}D_{HL}^\varep(\bv,\bomega )\\
=&\lim\limits_{\varep\to0}\f34\int \nabla \varphi^\varep(\ell)\cdot\B[  \delta \bv(\delta \bv_L\cdot \delta \bomega_L)-\f12\delta \bomega [\delta \bv_L]^2\B]d^3\bl\\&+\lim\limits_{\varep\to0}\f32\int \f{1}{|\bl|}\varphi^\varep\bn\cdot\B[  \delta \bv(\delta \bv_T\cdot \delta \bomega_T)-\f12\delta \bomega[\delta \bv_T]^2 \B]d^3\bl\\
&+\lim\limits_{\varep\to0}\f32\int \f{1}{|\bl|}\varphi^\varep\bn\cdot\B[\delta\bv\times (\delta \bomega \times \delta \bv)\B] d^3\bl\\
=&\lim\limits_{\varep\to0}\B(\f34\int_0^\infty r^3\varphi^{'}(r)dr 4\pi \bar {S}_{HL}(\bv,\bomega,\varep\ell)\B)+\lim\limits_{\varep\to0}\B(\f32 \int_0^\infty \varphi(r) r^2 dr 4\pi \bar{S}_{HT}(\bv,\bomega,\varep\ell)\B)\\
&+\lim\limits_{\varep\to0}\B(\f32\int_0^\infty r^2\varphi(r) dr 4\pi \bar{S}_{H}(\bv,\bomega,\varep\ell)\B)\\
=&-\f94 \bar {S}_{HL}(\bv,\bomega)+\f32 \bar{S}_{HT}(\bv,\bomega)+\f32\bar{S}_{H}(\bv,\bomega).
\ea\ee
A combination of  \eqref{2.57}  and \eqref{2.58}, gives,  as $\varep\to0$,
\be\label{2.59}\left\{\ba
&D_{H}(\bv,\bomega)=-\f{15}{8} \bar{S}_{HT}(\bv,\bomega)-\f34\bar{S}_H(\bv,\bomega), \\
&D_{H}(\bv,\bomega)=-\f94 \bar {S}_{HL}(\bv,\bomega)+\f32 \bar{S}_{HT}(\bv,\bomega)+\f32\bar{S}_{H}(\bv,\bomega),
 \ea\right.\ee
 which implies that
\be\label{2.63}\ba
D_{H}(\bv,\bomega)&=-\f54\bar{S}_{HL}(\bv,\bomega)+\f12\bar{S}_{H}(\bv,\bomega)\\
&=-\f54 S_{HL}(\bv,\bomega),
\ea\ee
where \be\ba
S_{HL}(\bv,\bomega)&=\lim_{\lambda\to 0}\f{1}{\lambda}\int_{\partial B } \bn \cdot \B[  \delta \bv(\lambda\bl)(\delta \bv_L(\lambda\bl)\cdot \delta \bomega_L(\lambda\bl))-\f12\delta \bomega(\lambda\bl)[\delta \bv_L(\lambda\bl)]^2\B] \f{d\sigma(\bl) }{4\pi}\\
&-\f25\lim_{\lambda\to 0}\f{1}{\lambda}\int_{\partial B } \bn \cdot\B[\delta\bv(\lambda\bl)\times (\delta \bomega(\lambda\bl) \times \delta \bv(\lambda\bl))\B] \f{d\sigma(\bl) }{4\pi}.
\ea\ee
Then, we have proved the 4/5 law and 8/15 law for the helicity of Euler equations.

\underline{{\bf Step 4}}: To complete the proof, it remains to
  to prove the equality  \eqref{2.27} we have assumed. Indeed, for any vectors ${\bf A}=(A_1,A_2,A_3), {\bf B}=(B_1,B_2,B_3)$ and ${\bf C}=(C_1,C_2,C_3)$, we conclude by \eqref{c11} that
\be\label{c26}\ba
&\B[\f{\partial}{\partial_{\ell_k}}(\bn_i\bn_j)-(\f{\partial \bn_i}{\partial\ell_j}+\f{\partial \bn_j}{\partial\ell_i})\bn_k\B]A_kB_iC_j\\
=&\B[\big(\f{\partial \bn_i}{\partial \ell_k}\bn_j+\f{\partial \bn_j}{\partial \ell_k}\bn_i\big)-\f{2}{|\bl|}(\delta^{ij}-\bn_i\bn_j)\bn_k\B]A_kB_iC_j\\
=&\f{1}{|\bl|}\B[(\delta^{ik}-\bn_i\bn_k)\bn_j+(\delta^{jk}-\bn_j\bn_k)\bn_i-2(\delta^{ij}-\bn_i\bn_j)\bn_k\B]A_kB_iC_j\\
=&\f{1}{|\bl|}\B[(\delta^{ik}\bn_j+\delta^{jk}\bn_i-2\delta^{ij}\bn_k  ]A_{k}B_{i}C_{j}.
\ea\ee
For the sake of convience, we denote the left-hand side of above equality by $I$. Then, we will discuss term $I$ into five cases.\\
Case 1:  $i=k=j$. For this case, we see that
\be\label{c27}
I=\f{1}{|\bl|}(\bn_j+\bn_i-2\bn_k)A_kB_iC_j={\bf 0}.
\ee
Case 2:  $i=k$ but $j\neq k$. We notice that
\be\label{c28}\ba
I&=\f{1}{|\bl|}\bn_jA_kB_iC_j\\
&=\f{1}{|\bl|}\B[\bn_1C_1(A_2B_2+A_3B_3)+\bn_2C_2(A_1B_1+A_3B_3)+\bn_3C_3(A_1B_1+A_2B_2)\B]\\
&=\f{1}{|\bl|}\B[\bn_1C_1 ({\bf A}\cdot {\bf B})-\bn_1A_1B_1C_1+\bn_2C_2 ({\bf A}\cdot {\bf B})-\bn_2A_2 B_2C_2+\bn_3C_3({\bf A}\cdot {\bf B})-\bn_3A_3B_3C_3\B]\\
&=\f{1}{|\bl|}\B[\bn\cdot {\bf C} ({\bf A}\cdot {\bf B})-\bn_1A_1B_1C_1-\bn_2A_2 B_2C_2-\bn_3A_3B_3C_3\B].
\ea\ee
Case 3:   $i\neq k$ but $i=j$. We compute
\be\label{c29}\ba
I=&\f{1}{|\bl|}(-2\bn_k)A_kB_iC_j\\
=&-\f{2}{|\bl|}\B[\bn_1A_1 (B_2C_2+B_3C_3)+\bn_2A_2 (B_1C_1+B_3C_3)+\bn_3A_3 (B_1C_1+B_2C_2)\B]\\
=&-\f{2}{|\bl|}\B[\bn\cdot {\bf A}({\bf B}\cdot {\bf C})-\bn_1A_1B_1C_1-\bn_2A_2 B_2C_2-\bn_3A_3B_3C_3\B].
\ea\ee
Case 4:   $i\neq k$ but $j=k$.  A straightforward computation leads to
\be\label{c30}\ba
I=&\f{1}{|\bl|}\bn_iA_kB_iC_j\\
=&\f{1}{|\bl|}\B[\bn_1B_1(A_2C_2+A_3C_3)+\bn_2B_2(A_1C_1+A_3C_3)+\bn_3B_3(A_1C_1+A_2C_2)\B]\\
=&\f{1}{|\bl|}\B[\bn\cdot {\bf B}({\bf A}\cdot {\bf C})-\bn_1A_1B_1C_1-\bn_2A_2 B_2C_2-\bn_3A_3B_3C_3\B].
\ea\ee
Case 5:  $i\neq k$, $j\neq k$ and $j\neq i$. It is clear that
\be\label{c31}\ba
I={\bf 0}.
\ea\ee
Thus, combining all above cases together, we have
\be\label{c32}\ba
&\B[\f{\partial}{\partial_{\ell_k}}(\bn_i\bn_j)-(\f{\partial \bn_i}{\partial\ell_j}+\f{\partial \bn_j}{\partial\ell_i})\bn_k\B]A_kB_iC_j\\
=&\f{1}{|\bl|}\bn\cdot\B[{\bf C}({\bf A}\cdot {\bf B})+{\bf B}({\bf A}\cdot {\bf C})-2{\bf A}({\bf B}\cdot {\bf C})\B].
\ea\ee
The proof of this theorem is now completed.
\end{proof}

\section{Four-fifths   laws of total energy and cross-helicity in magnetized fluids}
This section is to prove Theorem \ref{the1.2}.  We will consider      dissipation rates of total energy and cross-helicity in magnetohydrodynamics
flow, respectively.

\subsection{Total energy}
In this subsection, we  present 4/5 law of total energy for the MHD equations.
\begin{proof}[Proof of Theorem \ref{the1.2}: Energy]
On the one hand, following the same path of \eqref{c17}-\eqref{c19}, we can obtain
$$\left\{\ba
& \partial_{t} \bm{v}_{L}^{\varepsilon}\cdot\bm{v} +\text{div} (\bm{v}\otimes \bm{v}_{L})^{\varepsilon}\cdot\bm{v} -\text{div} (\bm{h}\otimes \bm{h}_{L})^{\varepsilon}\cdot\bm{v} +  \nabla\Pi_{L} ^{\varepsilon}\cdot\bm{v}=0, \\
&\partial_{t} \bm{v}\cdot\bm{v}_{L}^{\varepsilon} +\text{div} (\bm{v}\otimes \bm{v})\cdot\bm{v}_{L}^{\varepsilon} -\text{div} (\bm{h}\otimes \bm{h})\cdot\bm{v}_{L}^{\varepsilon} +  \nabla\Pi\cdot \bm{v}_{L} ^{\varepsilon}=0, \\&
\partial_{t} \bm{h}_{L}^{\varepsilon}\cdot\bm{h}+\text{div} (\bm{v}\otimes \bm{h}_{L})^{\varepsilon} \cdot\bm{h} -\text{div} (\bm{ h}\otimes \bm{ v}_{L})^{\varepsilon}\cdot\bm{h}=0,\\&
 \partial_{t} \bm{h}\cdot\bm{h}_{L}^{\varepsilon}+\text{div} (\bm{v}\otimes \bm{h}) \cdot\bm{h}^{\varepsilon}_{L} -\text{div} (\bm{ h}\otimes \bm{ v})\cdot\bm{h}_{L}^{\varepsilon}=0,\\
&\Div \bm{v}_{L}^{\varepsilon}=\Div \bm{h}_{L}^{\varepsilon}=0.
 \ea\right.$$
Then by using the Leibniz formula, it shows  that
  $$\ba
  \partial_{i}(v_{i}v_{L_{j}})^{\varepsilon}v_{j}+  \partial_{i}(v_{i}v_{j})v_{L_{j}}^{\varepsilon}
  =&\partial_{i}[(v_{i}v_{j})v_{L_{j}}^{\varepsilon}]
  +v_j\partial_{i}\B((v_{i}v_{L_{j}})^{\varepsilon}
  -(v_{i}v_{L_{j}}^{\varepsilon})\B),
\\
 -\B[\partial_{i}(h_{i}h_{L_{j}})^{\varepsilon}v_{j}+
 \partial_{i}(h_{i}h_{j})v_{L_{j}}^{\varepsilon}\B]
    =&-\partial_{i}[(h_{i}h_{j})v_{L_{j}}^{\varepsilon}]- \partial_{i}(h_{i}h_{L_{j}})^{\varepsilon}v_{j}+ h_{j}\partial_{i}(h_{i}v_{L_{j}}^{\varepsilon}),
\\
  \partial_{i}(v_{i}h_{L_{j}})^{\varepsilon}h_{j}+  \partial_{i}(v_{i}h_{j})h_{L_{j}}^{\varepsilon}
  =&\partial_{i}[(v_{i}h_{j})h_{L_{j}}^{\varepsilon}]
  +h_j\partial_{i}\B((v_{i}h_{L_{j}})^{\varepsilon}
  -(v_{i}h_{L_{j}}^{\varepsilon})\B),
\\
 -\B[\partial_{i}(h_{i}v_{L_{j}})^{\varepsilon}h_{j}+
 \partial_{i}(h_{i}v_{j})h_{L_{j}}^{\varepsilon}\B]
   =&-\partial_{i}[(h_{i}v_{j})h_{L_{j}}^{\varepsilon}]- \partial_{i}(h_{i}v_{L_{j}})^{\varepsilon}h_{j}+ v_{j}\partial_{i}(h_{i}h_{L_{j}}^{\varepsilon}).
  \ea $$
Hence, we arrive at
\be\ba\label{3.1}
&\partial_{t}(  \bm{v}_{L}^{\varepsilon}\cdot \bm{v} + \bm{h}\cdot \bm{h}_{L}^{\varepsilon})+\s\B[\Pi_L^\varep \bv+\Pi \bv_L^\varep\B]+
\s\B[\bm{v} (\bm{v}\cdot \bm{v}_{L }^{\varepsilon})
- \bm{h}(\bm{h}\cdot \bm{v}_{L }^{\varepsilon})
+ \bm{v}(\bm{h}\cdot \bm{h}_{L }^{\varepsilon})
- \bm{h}(\bm{v} \cdot \bm{h}_{L }^{\varepsilon})\B]
\\
=&-v_j \partial_{i}\B[(v_{i}v_{L_{j}})^{\varepsilon}
  -(v_{i}v_{L_{j}}^{\varepsilon})\B]
+v_j\partial_{i}\B[ (h_{i}h_{L_{j}})^{\varepsilon}- (h_{i}h_{L_{j}}^{\varepsilon})\B] \\&-h_j\partial_{i}\B[(v_{i}h_{L_{j}})^{\varepsilon}
  -(v_{i}h_{L_{j}}^{\varepsilon})\B]+
  h_{j}\partial_{i}\B[(h_{i}v_{L_{j}})^{\varepsilon}- (h_{i}v_{L_{j}}^{\varepsilon})\B].
   \ea \ee
 On the other hand, we derive from   \eqref{2.35} that
 \be\label{3.2}\ba
  &\f12\int_{\mathbb{T}^{3}} \nabla \varphi^\varep(\ell)\cdot \delta \bm{v} [\delta \bm{v}_L]^2+\f{2}{|\bl|}\varphi^\varep(\ell)\bn\cdot \delta \bm{v} [\delta \bm{v}_T]^2 d^3\bl
   \\& +\f12\s \B[\big(\bm{v}(\bm{v}_L\cdot\bm{v}_L)\big)^\varep-\bm{v}(\bm{v}_L\cdot\bm{v}_L)^\varep\B]
   =v_j\partial_k\B[(v_k v_{L_j})^\varep-(v_k v_{L_j}^\varep)\B],
   \ea\ee
and
 \be\label{3.3}\ba
&\f12\int_{\mathbb{T}^{3}} \nabla \varphi^\varep(\ell)\cdot \delta \bm{v} [\delta \bm{h}_L]^2+\f{2}{|\bl|}\varphi^\varep(\ell)\bn\cdot \B[\delta \bm{v} [\delta \bm{h}_T]^2 +\delta \bm{h}(\delta \bm{v} \cdot \delta \bm{h})-\delta \bm{v} (\delta \bm{h}\cdot \delta \bm{h})\B]d^3\bl\\&+\f12
\s \B[\big(\bm{v}(\bm{h}_L\cdot\bm{h}_L)\big)^\varep-\bm{v}(\bm{h}_L\cdot\bm{h}_L)^\varep\B]
\\=& h_j\partial_k\B[(v_k h_{L_j})^\varep-(v_k h_{L_j}^\varep)\B].
\ea\ee
Similarly, making use of \eqref{2.39}, we discover that
    \be\label{3.4}\ba
&\int_{\mathbb{T}^{3}} \nabla \varphi^\varep(\ell)\cdot \delta \bm{h}(\delta \bm{v}_L \cdot \delta \bm{h}_L)+\f{2}{|\bl|}\varphi^\varep(\ell)\bn\cdot \delta \bm{h}(\delta \bm{v}_T\cdot \delta\bm{h}_T)\\&- \f{1}{|\bl|}\varphi^\varep\bn\cdot\B[\delta \bm{h}(\delta \bm{h} \cdot \delta \bm{v})-\delta \bm{v} (\delta \bm{h}\cdot \delta \bm{h})\B] d^3 \bl +\s\B[\big(\bm{h} (\bm{v}_L\cdot \bm{h}_L)\big)^\varep-\bm{h}(\bm{v}_L\cdot \bm{h}_L)^\varep\B]\\&=h_j\partial_k\B[(h_k v_{L_j})^\varep-(h_k v_{L_j}^\varep)\B]+v_i\partial_k\B[(h_k h_{L_i})^\varep-(h_k h_{L_i}^\varep)\B].
\ea\ee
Plugging \eqref{3.2}-\eqref{3.4} into \eqref{3.1}, we end up with
  \be\label{3.3-1}\ba
 &\partial_{t}(  \bm{v}_{L}^{\varepsilon}\cdot \bm{v} + \bm{h}\cdot \bm{h}_{L}^{\varepsilon})+
 \s\B[\bm{v} (\bm{v}\cdot \bm{v}_{L }^{\varepsilon})
 - \bm{h}(\bm{h}\cdot \bm{v}_{L }^{\varepsilon})
 + \bm{v}(\bm{h}\cdot \bm{h}_{L }^{\varepsilon})
 - \bm{h}(\bm{v} \cdot \bm{h}_{L }^{\varepsilon})\B]
\\&+\s\B[\Pi_L^\varep \bv+\Pi \bv_L^\varep\B]
 -\s\B[\B(\bm{h} (\bm{v}_L\cdot \bm{h}_L)\B)^\varep-\bm{h}(\bm{v}_L\cdot \bm{h}_L)^\varep\B]\\& +\f12\s \B[\B(\bm{v}(\bm{v}_L\cdot\bm{v}_L)\B)^\varep-\bm{v}(\bm{v}_L\cdot\bm{v}_L)^\varep\B]+\f12
\s \B[\B(\bm{v}(\bm{h}_L\cdot\bm{h}_L)\B)^\varep
-\bm{v}(\bm{h}_L\cdot\bm{h}_L)^\varep\B]
\\
   =&-\f{2}{3}D_{EL}^{\varepsilon}(\bm{v},\bm{h}),
\ea\ee
where
$$\ba  &D_{EL}^{\varepsilon}(\bm{v},\bm{h})\\
=&\f34\int_{\mathbb{T}^{3}} \nabla \varphi^\varep(\ell)\cdot \delta \bm{v} \B([\delta \bm{v}_L]^2+[\delta \bm{h}_L]^2\B)+\f{2}{|\bl|}\varphi^\varep\bn\cdot \delta \bm{v} \B([\delta \bm{v}_T]^2+[\delta \bm{h}_T]^2\B) d^3\bl
\\
&-\f32\int_{\mathbb{T}^{3}} \nabla \varphi^\varep(\ell)\cdot \delta \bm{h}(\delta \bm{v}_L \cdot \delta \bm{h}_L)+\f{2}{|\bl|}\varphi^\varep\bn\cdot \delta \bm{h}(\delta \bm{v}_T\cdot \delta\bm{h}_T)d^3\bl\\&-3\int_{\mathbb{T}^3} \f{1}{|\bl|}\varphi^\varep\bn\cdot\big(\delta \bm{h}\times(\delta\bm{v}\times \delta\bm{h})\big) d^3 \bl,
  \ea$$
  and the indentity $\delta \bm{h}\times(\delta\bm{v}\times \delta\bm{h})=\delta \bm{v}(\delta \bm{h}\cdot \delta \bm{h})-\delta \bm{h}(\delta\bm{h}\cdot \delta\bm{v})$ has been used.

Proceeding similarly as  the derivation of \eqref{3.1}, we have
\be\ba \label{3.5}
&\partial_{t}(  \bm{v}_{T}^{\varepsilon}\cdot \bm{v} + \bm{h}\cdot \bm{h}_{T}^{\varepsilon})+\s\B[\Pi_T^\varep \bv+\Pi \bv_T^\varep\B]+
\s\B[\bm{v} (\bm{v}\cdot \bm{v}_{T }^{\varepsilon})
- \bm{h}(\bm{h}\cdot \bm{v}_{T }^{\varepsilon})
+ \bm{v}(\bm{h}\cdot \bm{h}_{T }^{\varepsilon})
- \bm{h}(\bm{v} \cdot \bm{h}_{T }^{\varepsilon})\B]\\
=&-v_j \partial_{i}\B[(v_{i}v_{T_{j}})^{\varepsilon}
-(v_{i}v_{T_{j}}^{\varepsilon})\B]
+v_j\partial_{i}\B[ (h_{i}h_{T_{j}})^{\varepsilon}- (h_{i}h_{T_{j}}^{\varepsilon})\B] \\&-h_j\partial_{i}\B[(v_{i}h_{T_{j}})^{\varepsilon}
-(v_{i}h_{T_{j}}^{\varepsilon})\B]+
h_{j}\partial_{i}\B[(h_{i}v_{T_{j}})^{\varepsilon}- (h_{i}v_{T_{j}}^{\varepsilon})\B].
   \ea \ee
In light of  \eqref{2.49},  it gives that
\be\label{3.6}\ba
&\int_{\mathbb{T}^{3}} \nabla \varphi^\varep(\ell)\cdot \delta \bm{h}(\delta \bm{v}_T\cdot \delta \bm{h}_T)-\f{2}{|\bl|}\varphi^\varep\bn\cdot \delta\bm{h} (\delta\bm{v}_T\cdot \delta\bm{h}_T)-\f{1}{|\bl|}\bn\cdot\big[\delta \bm{h}\times (\delta \bv\times \delta \bm{h})\big] d^3\bl \\&+\text{div} \B[\big(\bm{h}(\bm{v}_T\cdot \bm{h}_T)\big)^\varep-\bm{h}(\bm{v}_T\cdot \bm{h}_T)^\varep\B]\\
=&h_j\partial_k\B[(h_k v_{T_j})^\varep-(h_k v_{T_j}^\varep)\B]+v_i\partial_k\B[(h_k h_{T_i})^\varep-(h_k h_{T_i}^\varep)\B].
\ea\ee
It follows from \eqref{2.51} that
\be\label{3.7}\ba
&\f12\int_{\mathbb{T}^{3}} \nabla \varphi^\varep(\ell)\cdot \delta \bm{v}[\delta \bm{v}_T]^2-\f{2}{|\bl|}\varphi^\varep\bn\cdot \delta \bm{v} [\delta \bm{v}_T]^2  d^3\bl +\f12\s \B[\big(\bm{v}(\bm{v}_T\cdot\bm{v}_T)\big)^\varep-\bm{v}(\bm{v}_T\cdot\bm{v}_T)^\varep\B]
\\
=&v_j\partial_k\B[(v_k v_{T_j})^\varep-(v_k v_{T_j}^\varep)\B],
\ea\ee
and
\be\label{3.8}\ba
&\f12\int_{\mathbb{T}^{3}} \nabla \varphi^\varep(\ell)\cdot \delta \bm{v}[\delta \bm{h}_T]^2-\f{2}{|\bl|}\varphi^\varep\bn\cdot \delta \bm{v} [\delta \bm{h}_T]^2  +\f{2}{|\bl|}\bn\cdot\big[\delta \bm{h}\times (\delta \bv\times \delta \bm{h})\big]d^3\bl\\ &+\f12\s \B[\big(\bm{v}(\bm{h}_T\cdot\bm{h}_T)\big)^\varep-\bm{v}(\bm{h}_T\cdot\bm{h}_T)^\varep\B]
\\
=&h_j\partial_k\B[(v_k h_{T_j})^\varep-(v_k h_{T_j}^\varep)\B].
\ea\ee
Inserting \eqref{3.6}-\eqref{3.8} into \eqref{3.5}, we arrive at
\be\label{3.3-2}\ba &\partial_{t}(  \bm{v}_{T}^{\varepsilon}\cdot \bm{v} + \bm{h}\cdot \bm{h}_{T}^{\varepsilon})+
\s\B[\bm{v} (\bm{v}\cdot \bm{v}_{T }^{\varepsilon})
- \bm{h}(\bm{h}\cdot \bm{v}_{T }^{\varepsilon})
+ \bm{v}(\bm{h}\cdot \bm{h}_{T }^{\varepsilon})
- \bm{h}(\bm{v} \cdot \bm{h}_{T }^{\varepsilon})\B]
\\&+\s\B[\Pi_T^\varep \bv+\Pi \bv_T^\varep\B]-\text{div} \B[\B(\bm{h}(\bm{v}_T\cdot \bm{h}_T)\B)^\varep- \bm{h}(\bm{v}_T\cdot \bm{h}_T)^\varep\B]\\& +\f12\s \B[\B(\bm{v}(\bm{v}_T\cdot\bm{v}_T)\B)^\varep-\bm{v}(\bm{v}_T\cdot\bm{v}_T)^\varep\B]+\f12\s \B[\B(\bm{v}(\bm{h}_T\cdot\bm{h}_T)\B)^\varep
-\bm{v}(\bm{h}_T\cdot\bm{h}_T)^\varep\B]\\ 
=&-\f43D^{\varepsilon}_{ET}(\bm{v},\bm{h}),
   \ea \ee
where
$$\ba
D^{\varepsilon}_{ET}(\bm{v},\bm{h})=&\f{3}{8}\int_{\mathbb{T}^{3}} \nabla \varphi^\varep(\ell)\cdot \delta \bm{v}\B[(\delta \bm{v}_T)^2+(\delta \bm{h}_T)^2\B]-\f{2}{|\bl|}\varphi^\varep\bn\cdot \delta \bm{v} \B[(\delta \bm{v}_T)^2 +(\delta \bm{h}_T)^2\B] d^3\bl\\  &-\f34\int_{\mathbb{T}^{3}} \nabla \varphi^\varep(\ell)\cdot \delta \bm{h}(\delta \bm{v}_T\cdot \delta \bm{h}_T)-\f{2}{|\bl|}\varphi^\varep\bn\cdot \delta\bm{h} (\delta\bm{v}_T\cdot \delta\bm{h}_T)d^3\bl\\
&+\f32\int_{\mathbb{T}^3}\f{1}{|\bl|}\bn\cdot
\big[\delta \bm{h}\times (\delta \bv\times \delta \bm{h})\big] d^3\bl.
\ea$$
An argument similar to the proof of Theorem \ref{the01} in Step 2 shows that the left-hand side of \eqref{3.3-1} and \eqref{3.3-2} converge to
\begin{equation}
	\f13\B[\partial_t([\bv]^2+[\bm{h}]^2)+\s(\Pi \bv)+\s\B(\bv[\bv]^2-2\bm{h}(\bm{h}\cdot\bv)+\bv[\bm{h}]^2\B)\B]=:-\f23 D_E(\bv,\bm{h}),
\end{equation}
and
\begin{equation}
	\f23\B[\partial_t([\bv]^2+[\bm{h}]^2)+\s(\Pi \bv)+\s\B(\bv[\bv]^2-2\bm{h}(\bm{h}\cdot\bv)+\bv[\bm{h}]^2\B)\B]=:-\f43 D_E(\bv,\bm{h}),
\end{equation}
respectively, in the sense of distributions. Namely, we have
\be
 D^\varep_{EX}(\bv,\bm{h})\to D_{E}(\bv,\bm{h}),\ee
for both $X=L,T$, as $\varep\to 0$, in the sense of distributions.

To proceed further, we set
$$\ba
&\overline{S}_{EL}(\bm{v}, \bm{h},\lambda)
=\f{1}{\lambda}\int_{\partial B } \bn \cdot \B[  \delta \bm{v} (\lambda\bl)\big([\delta \bm{v}_L(\lambda\bl)]^2+[\delta \bm{h}_L(\lambda\bl)]^2\big)-2\delta \bm{h}(\lambda\bl)\big(\delta \bm{v}_L (\lambda\bl)\cdot \delta \bm{h}_L(\lambda\bl)\big)\B] \f{d\sigma(\bl) }{4\pi},\\
&\overline{S}_{ET}(\bm{v}, \bm{h},\lambda )=\f{1}{\lambda}\int_{\partial B } \bn \cdot \B[  \delta \bm{v} (\lambda\bl)\big([\delta \bm{v}_T(\lambda\bl)]^2+[\delta \bm{h}_T(\lambda\bl)]^2\big)-2\delta \bm{h}(\lambda\bl)\big(\delta \bm{v}_T (\lambda\bl)\cdot \delta \bm{h}_T(\lambda\bl)\big)\B] \f{d\sigma(\bl) }{4\pi},\\
&\overline{S}_{E} (\bm{v}, \bm{h},\lambda)=\f{1}{\lambda}\int_{\partial B } \bn \cdot\B[\delta \bm{h}(\lambda\bl)\times\big(\delta\bm{v}(\lambda\bl)\times \delta\bm{h}(\lambda\bl)\big)\B] \f{d\sigma(\bl) }{4\pi}.
\ea$$
A simple computation leads to
\be\ba\label{3.9}
D_{E}(\bm{v},\bm{h})=&\lim_{\varepsilon\rightarrow0}D^{\varepsilon}_{ET}(\bm{v},\bm{h})\\
=&\lim_{\varepsilon\rightarrow0}\f{3}{8}\int_{\mathbb{T}^{3}}\B[\nabla \varphi^\varep(\ell)-\f{2}{|\bl|}\varphi^\varep\bn\B]\cdot \B[\delta \bm{v} [\delta \bm{v}_T]^2+   \delta \bm{v} [\delta \bm{h}_T]^2-2 \delta\bm{h} (\delta\bm{v}_T\cdot \delta\bm{h}_T)\B]   d^3\bl\\ &+\f32\lim_{\varepsilon\rightarrow0}\int_{\mathbb{T}^{3}}\f{1}{|\bl|}\bn\cdot
[\delta \bm{h}\times(\delta\bm{v}\times \delta\bm{h})]
d^3\bl \\ 
=&-\f{15}{8}\B(\bar{S}_{ET}(\bm{v}, \bm{h})-\f45\bar{S}_{E}(\bm{v}, \bm{h})\B)\\
=&-\f{15}{8} S_{ET}(\bm{v}, \bm{h}),
\ea\ee
and
$$\ba
D_{E}(\bm{v},\bm{h})=&\lim_{\varepsilon\rightarrow0}D_{EL}^{\varepsilon}(\bm{v},\bm{h})
\\ 
 =&\f34\lim_{\varepsilon\rightarrow0}\int_{\mathbb{T}^{3}} \nabla \varphi^\varep(\ell)\cdot \B[\delta \bm{v} \big([\delta \bm{v}_L]^2+[\delta \bm{h}_L]^2\big)-2\delta \bm{h}(\delta \bm{v}_L \cdot \delta \bm{h}_L)\B]d^3\bl\\&+\f34\lim_{\varepsilon\rightarrow0}\int_{\mathbb{T}^{3}}\f{2}{|\bl|}\varphi^\varep\bn\cdot \B[\delta \bm{v} \big([\delta \bm{v}_T]^2+[\delta \bm{h}_T]^2\big)-2 \delta \bm{h}(\delta \bm{v}_T\cdot \bm{h}_T)\B] d^3\bl\\
 &-3 \lim_{\varepsilon\rightarrow0}\int_{\mathbb{T}^{3}} \f{1}{|\bl|}\varphi^\varep\bn\cdot\B[\delta \bm{h}\times(\delta \bm{v} \times\delta \bm{h}) \B] d^3 \bl\\=&\f34\int_0^\infty r^3\varphi^{'}(r)dr 4\pi \bar {S}_{EL}(\bm{v}, \bm{h})+\f34 \lim_{\varepsilon\rightarrow0}\int_0^\infty 2\varphi(r) r^2 dr 4\pi \bar{S}_{ET}(\bm{v}, \bm{h})\\&-3\int_0^\infty r^2\varphi(r) dr 4\pi \bar{S}_{E}(\bm{v}, \bm{h})\\
=&-\f94 \bar {S}_{EL}(\bm{v}, \bm{h})+\f32 \bar{S}_{ET}(\bm{v}, \bm{h})-3\bar{S}_{E}(\bm{v}, \bm{h}).
  \ea$$
This together with \eqref{3.9} yields that
\be\label{3.10}\left\{\ba
&D_{E}(\bv,\bm{h})=-\f{15}{8} \bar{S}_{ET}(\bm{v}, \bm{h})+\f32\bar{S}_{E}(\bm{v}, \bm{h}),\\
&D_{E}(\bv,\bm{h})=-\f94 \bar {S}_{EL}(\bm{v}, \bm{h})+\f32 \bar{S}_{ET}(\bm{v}, \bm{h})-3\bar{S}_{E}(\bm{v}, \bm{h}),
 \ea\right.\ee
which turns out that
\be\label{3.11}\ba
D_{E}(\bv,\bm{h})=-\f54\bar{S}_{EL}(\bm{v}, \bm{h})-\bar{S}_{E}(\bm{v}, \bm{h})=-\f54\B(\bar{S}_{EL}(\bm{v}, \bm{h})+\f45\bar{S}_E(\bm{v}, \bm{h})\B)=-\f54 S_{EL}.
\ea\ee
 We derive from \eqref{3.9} and \eqref{3.11} that  $$S_{ET}(\bv,\bm{h})=-\f{8}{15}D_{E}(\bv,\bm{h}),~~S_{EL}(\bv,\bm{h})=-\f{4}{5}D_{E}(\bv,\bm{h}),$$
 where
 $$\ba S_{EL}&= \lim_{\lambda \to 0}\f{1}{\lambda}\int_{\partial B } \bn \cdot \B[  \delta \bm{v} (\lambda\bl)\big([\delta \bm{v}_L(\lambda\bl)]^2+[\delta \bm{h}_L(\lambda\bl)]^2\big)-2\delta \bm{h}(\lambda\bl)\big(\delta \bm{v}_L (\lambda\bl)\cdot \delta \bm{h}_L(\lambda\bl)\big)\B] \f{d\sigma(\bl) }{4\pi}\\
 &+\f45\lim_{\lambda \to 0}\f{1}{\lambda}\int_{\partial B } \bn \cdot\B[\delta \bm{h}(\lambda\bl)\times\big(\delta\bm{v}(\lambda\bl)\times \delta\bm{h}(\lambda\bl)\big)\B] \f{d\sigma(\bl) }{4\pi},\ea$$
 and
 $$\ba
 S_{ET}&=\lim_{\lambda \to 0}\f{1}{\lambda}\int_{\partial B } \bn \cdot \B[  \delta \bm{v} (\lambda\bl)\big([\delta \bm{v}_T(\lambda\bl)]^2+[\delta \bm{h}_T(\lambda\bl)]^2\big)-2\delta \bm{h}(\lambda\bl)\big(\delta \bm{v}_T (\lambda\bl)\cdot \delta \bm{h}_T(\lambda\bl)\big)\B] \f{d\sigma(\bl) }{4\pi}\\
 &-\f45 \lim_{\lambda \to 0}\f{1}{\lambda}\int_{\partial B } \bn \cdot\B[\delta \bm{h}(\lambda\bl)\times\big(\delta\bm{v}(\lambda\bl)\times \delta\bm{h}(\lambda\bl)\big)\B] \f{d\sigma(\bl) }{4\pi}.
 \ea$$
 Thus,
the proof of four-fifths   laws of total energy  in inviscid MHD equations is completed.
\end{proof}
\subsection{Cross-helicity  }
In this subsection, we consider Kolmogorov' 4/5 law  of cross-helicity in terms of mixed  velocity field and  magnetic field in MHD turbulence.

\begin{proof}[Theorem \ref{the1.2}: Cross-helicity]

A slight modification of the proof in \eqref{c17}-\eqref{c19} gives that
$$\left\{\ba
&  \partial_{t} \bm{v}_{L}^{\varepsilon}\cdot\bm{h} +\text{div} (\bm{v}\otimes \bm{v}_{L})^{\varepsilon}\cdot\bm{h} -\text{div} (\bm{h}\otimes \bm{h}_{L})^{\varepsilon}\cdot\bm{h} +  \nabla\Pi_{L} ^{\varepsilon}\cdot\bm{h}=0, \\
&
\partial_{t} \bm{h}\cdot\bm{v}_{L}^{\varepsilon}+\text{div} (\bm{v}\otimes \bm{h}) \cdot\bm{v}^{\varepsilon}_{L} -\text{div} (\bm{ h}\otimes \bm{ v})\cdot\bm{v}_{L}^{\varepsilon}=0,\\
&
 \partial_{t} \bm{h}_{L}^{\varepsilon}\cdot\bm{v}+\text{div} (\bm{v}\otimes \bm{h}_{L})^{\varepsilon} \cdot\bm{v} -\text{div} (\bm{ h}\otimes \bm{ v}_{L})^{\varepsilon}\cdot\bm{v}=0,\\
 & \partial_{t} \bm{v}\cdot\bm{h}_{L}^{\varepsilon} +\text{div} (\bm{v}\otimes \bm{v})\cdot\bm{h}_{L}^{\varepsilon} -\text{div} (\bm{h}\otimes \bm{h})\cdot\bm{h}_{L}^{\varepsilon} +  \nabla\Pi\cdot\bm{h}_{L} ^{\varepsilon}=0, \\
&\Div \bm{v}_{L}^{\varepsilon}=\Div \bm{h}_{L}^{\varepsilon}=0.
 \ea\right.$$
  It is easy to check that
  $$\ba
  &\partial_{i}(v_{i}v_{L_{j}})^{\varepsilon}h_{j}+  \partial_{i}(v_{i}h_{j})v_{L_{j}}^{\varepsilon}
   =\partial_i(v_ih_j v_{L_j}^\varep)+h_j\partial_{i}\B[(v_{i}v_{L_{j}})^{\varepsilon}-(v_i v_{L_j}^\varep)\B],\\
& -\B[\partial_{i}(h_{i}h_{L_{j}})^{\varepsilon}h_{j}+
 \partial_{i}(h_{i}v_{j})v_{L_{j}}^{\varepsilon}\B]
    =-\partial_{i}(h_{i}v_{j}v_{L_{j}}^{\varepsilon})+ v_j\partial_{i}(h_{i}v_{L_{j}}^{\varepsilon})- h_{j}\partial_{i}(h_{i}h_{L_{j}})^{\varepsilon},\\
   &\partial_{i}(v_{i}h_{L_{j}})^{\varepsilon}v_{j}+  \partial_{i}(v_{i}v_{j})h_{L_{j}}^{\varepsilon}
 =\partial_{i}(v_{i}v_{j}h_{L_{j}}^{\varepsilon})
  +v_j\partial_{i}\B[(v_{i}h_{L_{j}})^{\varepsilon}
  -(v_{i}h_{L_{j}}^{\varepsilon})\B],
 \\
  &-\B[\partial_{i}(h_{i}v_{L_{j}})^{\varepsilon}v_{j}+
 \partial_{i}(h_{i}h_{j})h_{L_{j}}^{\varepsilon}\B]
   =-\partial_{i}(h_{i}h_{j}h_{L_{j}}^{\varepsilon})+ h_j\partial_{i}(h_{i}h_{L_{j}}^{\varepsilon})- v_{j}\partial_{i}(h_{i}v_{L_{j}})^{\varepsilon}.  \ea $$
This in turn means that
\be\ba\label{4.1}
& \partial_{t} ( {\bv}_{L}^{\varepsilon}\cdot \bm{h}+ {\bv}\cdot \bm{h}_{L}^{\varepsilon} )+\s(\Pi^\varep_L \bm{h}+\bm{h}_L^\varep \Pi)+\s
\B[\bv(\bm{h}\cdot\bm{v}_L^\varep)
+\bv(\bv\cdot\bm{h}_L^\varep)-\bm{h}(\bv\cdot\bv_L^\varep)
-\bm{h}(\bm{h}\cdot\bm{h}_L^\varep)\B]\\
  =&-h_j\partial_{i}\B[(v_{i}v_{L_{j}})^{\varepsilon}-(v_i v_{L_j}^\varep)\B]-v_j\partial_{i}\B[(v_{i}h_{L_{j}})^{\varepsilon}
  -(v_{i}h_{L_{j}}^{\varepsilon})\B]+h_j\partial_i\B[(h_ih_{L_j})^\varep-(h_i h_{L_j}^\varep)\B]\\
  &+v_j\partial_i\B[(h_i v_{L_j})^\varep-(h_i v_{L_j}^\varep)\B].
    \ea \ee
Repeating the derivation of \eqref{2.35}, we have
\be\label{4.2}\ba
  &\f12\int_{\mathbb{T}^{3}} \nabla \varphi^\varep(\ell)\cdot \delta \bm{h} [\delta \bm{h}_L]^2+\f{2}{|\bl|}\varphi^\varep\bn\cdot \delta \bm{h} [\delta \bm{h}_T]^2 d^3\bl
 \\&+\f12\s \B[\big(\bm{h}(\bm{h}_L\cdot\bm{h}_L)\big)^\varep-\bm{h}(\bm{h}_L\cdot\bm{h}_L)^\varep\B]
  \\=& h_j\partial_k\B[(h_k h_{L_j})^\varep-(h_k h_{L_j}^\varep)\B],
   \ea\ee
and
 \be\label{4.3}\ba
&\f12\int_{\mathbb{T}^{3}} \nabla \varphi^\varep(\ell)\cdot \delta \bm{h} [\delta \bm{v}_L]^2+\f{2}{|\bl|}\varphi^\varep\bn\cdot \B[\delta \bm{h}[\delta \bm{v}_T]^2 +\delta \bm{v}(\delta \bm{h} \cdot \delta \bm{v})-\delta \bm{h} (\delta \bm{v}\cdot \delta \bm{v})\B]d^3\bl\\&+\f12
\s \B[\big(\bm{h}(\bm{v}_L\cdot\bm{v}_L)\big)^\varep-\bm{h}(\bm{v}_L\cdot\bm{v}_L)^\varep\B]
\\=& v_j\partial_k\B[(h_k v_{L_j})^\varep-(h_k v_{L_j}^\varep)\B].
\ea\ee
Moreover, along the same line of \eqref{2.39}, we get
\be\label{4.4}\ba
&\int_{\mathbb{T}^{3}} \nabla \varphi^\varep(\ell)\cdot \delta \bm{v}(\delta \bm{h}_L \cdot \delta \bm{v}_L)+\f{2}{|\bl|}\varphi^\varep\bn\cdot \delta \bm{v}(\delta \bm{h}_T\cdot \delta\bm{v}_T)\\&- \f{1}{|\bl|}\varphi^\varep\bn\cdot\B[\delta \bm{v}(\delta \bm{v} \cdot \delta \bm{h})-\delta \bm{h} (\delta \bm{v}\cdot \delta \bm{v})\B] d^3 \bl +\s\B[\big(\bv (\bm{h}_L\cdot \bm{v}_L)\big)^\varep-\bv(\bm{h}_L\cdot \bm{v}_L)^\varep\B]\\&=h_j\partial_k\B[(v_k v_{L_j})^\varep-(v_k v_{L_j}^\varep)\B]+v_i\partial_k\B[(v_k h_{L_i})^\varep-(v_k h_{L_i}^\varep)\B].
\ea\ee
Inserting \eqref{4.2}-\eqref{4.4} into \eqref{4.1}, we conclude that
\be\label{4.5}\ba
&  \partial_{t} ( {\bv}_{L}^{\varepsilon}\cdot \bm{h}+ {\bv}\cdot \bm{h}_{L}^{\varepsilon} )+\s\B(\bv(\bm{h}\cdot\bm{v}_L^\varep)+\bv(\bv\cdot\bm{h}_L^\varep)-\bm{h}(\bv\cdot\bv_L^\varep)-\bm{h}(\bm{h}\cdot\bm{h}_L^\varep)\B)\\
&+\s(\Pi^\varep_L \bm{h}+\bm{h}_L^\varep \Pi)+\s\B[\big(\bv (\bm{h}_L\cdot \bm{v}_L)\big)^\varep-\bv(\bm{h}_L\cdot \bm{v}_L)^\varep\B]
\\&-\f12\s\B[\big(\bm{h}(\bm{h}_L\cdot\bm{h}_L)\big)^\varep -\bm{h}(\bm{h}_L\cdot\bm{h}_L)^\varep\B]-\f12\s\B[\big(\bm{h}(\bm{v}_L\cdot\bm{v}_L)\big)^\varep-\bm{h}(\bm{v}_L\cdot\bm{v}_L)^\varep\B]\\
   =&-\f23D_{CHL}^{\varepsilon}(\bm{v},\bm{h})
\ea\ee
where
$$\ba
&D_{CHL}^{\varepsilon}(\bm{v},\bm{h})\\
=&\f32\int_{\mathbb{T}^{3}} \nabla \varphi^\varep(\ell)\cdot \delta \bm{v}(\delta \bm{h}_L \cdot \delta \bm{v}_L)+\f{2}{|\bl|}\varphi^\varep\bn\cdot \B[\delta \bm{v}(\delta \bm{h}_T\cdot \delta \bm{v}_T)+\delta \bm{v}\times(\delta \bm{h} \times \delta \bm{v})\B] d^3 \bl \\&-  \f{3}{4}\int_{\mathbb{T}^{3}} \nabla \varphi^\varep(\ell)\cdot \delta \bm{h}\B( [\delta \bm{h}_L]^2+[\delta \bm{v}_L]^2\B)+\f{2}{|\bl|}\varphi^\varep\bn\cdot \delta \bm{h} \B([\delta \bm{h}_T]^2+[\delta \bm{b}_T]^2\B) d^3\bl,
\ea$$
and the indentity $\delta \bm{v}\times(\delta \bm{h} \times \delta \bm{v})=\delta \bm{h} (\delta \bm{v}\cdot \delta \bm{v})-\delta \bm{v}(\delta \bm{h} \cdot \delta \bm{v})$ has been used.

On the other hand, following the path of \eqref{4.1}, we infer that
\be\label{4.6}\ba
 & \partial_{t} ( {\bv}_{T}^{\varepsilon}\cdot \bm{h}+ {\bv}\cdot \bm{h}_{T}^{\varepsilon} )+\s(\Pi^\varep_T \bm{h}+\bm{h}_T^\varep \Pi)+\s\B[\bv(\bm{h}\cdot\bm{v}_T^\varep)
  +\bv(\bv\cdot\bm{h}_T^\varep)-\bm{h}(\bv\cdot\bv_T^\varep)
  -\bm{h}(\bm{h}\cdot\bm{h}_T^\varep)\B]\\
  =&-h_j\partial_{i}\B[(v_{i}v_{T_{j}})^{\varepsilon}-(v_i v_{T_j}^\varep)\B]-v_j\partial_{i}\B[(v_{i}h_{T_{j}})^{\varepsilon}
  -(v_{i}h_{T_{j}}^{\varepsilon})\B]+h_j\partial_i\B[(h_ih_{T_j})^\varep-(h_i h_{T_j}^\varep)\B]\\
  &+v_j\partial_i\B[(h_i v_{T_j})^\varep-(h_i v_{T_j}^\varep)\B].
    \ea \ee
Arguing as \eqref{2.51}, we remark that
  \be\label{4.7}\ba
&\f12\int_{\mathbb{T}^{3}} \nabla \varphi^\varep(\ell)\cdot \delta \bm{h}[\delta \bm{h}_T]^2-\f{2}{|\bl|}\varphi^\varep\bn\cdot \delta \bm{h} [\delta \bm{h}_T]^2 d^3\bl\\ &+\f12\s \B[\big(\bm{h}(\bm{h}_T\cdot\bm{h}_T)\big)^\varep-\bm{h}(\bm{h}_T\cdot\bm{h}_T)^\varep\B]
\\
=&h_j\partial_k\B[(h_k h_{T_j})^\varep-(h_k h_{T_j}^\varep)\B],
\ea\ee
and
\be\label{4.8}\ba
&\f12\int_{\mathbb{T}^{3}} \nabla \varphi^\varep(\ell)\cdot \delta \bm{h}[\delta \bm{v}_T]^2-\f{2}{|\bl|}\varphi^\varep\bn\cdot \delta \bm{h} [\delta \bm{v}_T]^2 +\f{2}{|\bl|}\varphi^\varep\bn\cdot\big[ \delta \bv \times (\delta \bm{h}\times \delta \bv)\big]d^3\bl\\ &+\f12\s \B[\big(\bm{h}(\bm{v}_T\cdot\bm{v}_T)\big)^\varep-\bm{h}(\bm{v}_T\cdot\bm{v}_T)^\varep\B]
\\
=&v_j\partial_k\B[(h_k v_{T_j})^\varep-(h_k v_{T_j}^\varep)\B].
\ea\ee
By suitable modification of \eqref{2.49}, one arrives at
  \be\label{4.9}\ba
&\int_{\mathbb{T}^{3}} \nabla \varphi^\varep(\ell)\cdot \delta \bm{v}(\delta \bm{h}_T\cdot \delta \bm{v}_T)-\f{2}{|\bl|}\varphi^\varep\bn\cdot \delta\bm{v} (\delta\bm{h}_T\cdot \delta\bm{v}_T)-\f{1}{|\bl|}\varphi^\varep\bn\cdot \big[\delta \bv \times (\delta \bm{h}\times \delta\bv)\big] d^3\bl \\&+\text{div} \B[\big(\bm{v}(\bm{h}_T\cdot \bm{v}_T)\big)^\varep-\bm{v}(\bm{h}_T\cdot \bm{v}_T)^\varep\B]\\
=&v_j\partial_k\B[(v_k h_{T_j})^\varep-(v_k h_{T_j}^\varep)\B]+h_i\partial_k\B[(v_k v_{T_i})^\varep-(v_k v_{T_i}^\varep)\B].
\ea\ee
Substituting \eqref{4.7}-\eqref{4.9} into  \eqref{4.6}, it follows that
 \be\label{4.9-1}\ba & \partial_{t} ( {\bv}_{T}^{\varepsilon}\cdot \bm{h}+ {\bv}\cdot \bm{h}_{T}^{\varepsilon} )+\s\B[\bv(\bm{h}\cdot\bm{v}_T^\varep)+\bv(\bv\cdot\bm{h}_T^\varep)-\bm{h}(\bv\cdot\bv_T^\varep)-\bm{h}(\bm{h}\cdot\bm{h}_T^\varep)\B]\\
&+\s(\Pi^\varep_T \bm{h}+\bm{h}_T^\varep \Pi)+\text{div} \B[\big(\bm{v}(\bm{h}_T\cdot \bm{v}_T)\big)^\varep-\bm{v}(\bm{h}_T\cdot \bm{v}_T)^\varep\B]\\
&-\f12\s \B[\big(\bm{h}(\bm{h}_T\cdot\bm{h}_T)\big)^\varep-\bm{h}(\bm{h}_T\cdot\bm{h}_T)^\varep\B]-\f12\s \B[\big(\bm{h}(\bm{v}_T\cdot\bm{v}_T)\big)^\varep-\bm{h}(\bm{v}_T\cdot\bm{v}_T)^\varep\B]\\
  =&-\f43D_{CHT}^{\varepsilon}(\bm{v},\bm{h}),
    \ea \ee
where
  $$\ba
&D_{CHT}^{\varepsilon}(\bm{v},\bm{h})\\
     =&\f34\int_{\mathbb{T}^{3}} \nabla \varphi^\varep(\ell)\cdot \delta \bm{v}(\bl)(\delta \bm{h}_T\cdot \delta \bm{v}_T)-\f{2}{|\bl|}\varphi^\varep\bn\cdot \B[\delta\bm{v} (\delta\bm{h}_T\cdot \delta\bm{v}_T)+\delta \bm{v}\times(\delta\bm{h}\times \delta\bm{v}) \B] d^3\bl \\
  &- \f{3}{8}\int_{\mathbb{T}^{3}} \nabla \varphi^\varep(\ell)\cdot \delta \bm{h}\B([\delta \bm{h}_T]^2+[\delta \bv_T]^2\B)-\f{2}{|\bl|}\varphi^\varep\bn\cdot \delta \bm{h} \B([\delta \bm{h}_T]^2 +[\delta \bv_T]^2\B)d^3\bl.
  \ea $$

Next, proceeding  similarly as \eqref{2.54}-\eqref{2.57-1} in Section 2 shows that the left-hand side of \eqref{4.5} and \eqref{4.9-1} converge to
\begin{equation}
	\f23\B[\partial_t(\bv \cdot \bm{h})+\s(\Pi \bv)+\s\B(\bv(\bv\cdot \bm{h})-\f12 \bm{h}|\bv]^2-\f12\bm{h}|\bm{h}]^2\B)\B]=:-\f23 D_{CH}(\bv,\bm{h}),
\end{equation}
and
\begin{equation}
	\f43\B[\partial_t(\bv \cdot \bm{h})+\s(\Pi \bv)+\s\B(\bv(\bv\cdot \bm{h})-\f12 \bm{h}|\bv]^2-\f12\bm{h}|\bm{h}]^2\B)\B]=:-\f43 D_{CH}(\bv,\bm{h})
\end{equation}
respectively, in the sense of distributions. In other words, we have
\be
D^\varep_{CHX}(\bv,\bm{h})\to D_{CH}(\bv,\bm{h}),\ee
for both $X=L,T$, as $\varep\to 0$, in the sense of distributions.

Finally, to show the scaling law of cross-helicity, we let
$$\ba
&\overline{S}_{CHL}(\bm{v}, \bm{h})\\=&\lim_{\lambda \to0}\f{1}{\lambda}\int_{\partial B } \bn \cdot \B[ 2\delta\bm{v}(\lambda\bl) (\delta\bm{h}_L(\lambda\bl)\cdot \delta\bm{v}_L(\lambda\bl))- \delta \bm{h}(\lambda\bl) [\delta \bm{v}_L(\lambda\bl)]^2- \delta \bm{h} (\lambda\bl)[\delta \bm{h}_L(\lambda\bl)]^2\B] \f{d\sigma(\bl) }{4\pi},\\
&\overline{S}_{CHT}(\bm{v}, \bm{h})\\=&\lim_{\lambda \to0}\f{1}{\lambda}\int_{\partial B } \bn \cdot \B[ 2\delta\bm{v}(\lambda\bl) (\delta\bm{h}_T(\lambda\bl)\cdot \delta\bm{v}_T(\lambda\bl))- \delta \bm{h}(\lambda\bl) [\delta \bm{v}_T(\lambda\bl)]^2- \delta \bm{h} (\lambda\bl)[\delta \bm{h}_T(\lambda\bl)]^2\B] \f{d\sigma(\bl) }{4\pi},\\
&\overline{S}_{ CH} (\bm{v}, \bm{h})\\=&\lim_{\lambda \to0}\f{1}{\lambda}\int_{\partial B } \bn \cdot\B[\delta \bm{v}(\lambda\bl)\times(\delta\bm{h}(\lambda\bl)\times \delta\bm{v}(\lambda\bl))\B] \f{d\sigma(\bl) }{4\pi}.
\ea$$
  By a straightforward computation,  it follows that
   \be\ba\label{4.10}
 &D_{CH} (\bm{v},\bm{h})
 =\lim_{\varepsilon\rightarrow0}D_{CHT}^{\varepsilon}(\bm{v},\bm{h})\\
 =&\f38\lim_{\varepsilon\rightarrow0}\int_{\mathbb{T}^{3}}\B[ \nabla \varphi^\varep(\ell) -\f{2}{|\bl|}\varphi^\varep\bn\B] \cdot\B[2\delta\bm{v} (\delta\bm{h}_T\cdot \delta\bm{v}_T)- \delta \bm{h} [\delta \bm{v}_T]^2- \delta \bm{h} [\delta \bm{h}_T]^2 \B]d^3\bl\\&-\f32\lim_{\varepsilon\rightarrow0}\int_{\mathbb{T}^{3}}\f{1}{|\bl|}\varphi^\varep\bn\cdot\B[\delta \bm{v}\times(\delta\bm{h}\times\delta\bm{v}) \B] d^3\bl    \\=&\f38\B(\int_0^\infty r^3\varphi^{'}(r)-2r^2\varphi(r)dr\B)4\pi \bar{S}_{CHT}(\bm{v}, \bm{h})-\f32\int_0^\infty r^2\varphi(r)dr 4\pi\bar{S}_{CH}(\bm{v}, \bm{h})\\ 
=&-\f{15}{8}\B(\bar{S}_{CHT}(\bm{v}, \bm{h})+\f45\bar{S}_{CH}(\bm{v}, \bm{h})\B)\\
=&-\f{15}{8} S_{CHT}(\bm{v}, \bm{h}),
  \ea \ee
and
\be
\begin{aligned}
 &D_{CH} (\bm{v},\bm{h})
 =\lim_{\varep\to0}D_{CHL} ^\varep(\bm{v},\bm{h}) \\
    =&\f34\lim_{\varep\to0}\int_{\mathbb{T}^{3}} \nabla \varphi^\varep(\ell)\cdot\B[2\delta \bm{v}(\delta \bm{h}_L \cdot \delta \bm{v}_L)-\delta \bm{h} [\delta \bm{h}_L]^2-\delta \bm{h} [\delta \bm{v}_L]^2\B] d^3 \bl \\
    &+\f34\lim_{\varep\to0}\int_{\mathbb{T}^{3}}\f{2}{|\bl|}\varphi^\varep(\ell)\bn\cdot \B[2\delta \bm{v}(\delta \bm{h}_T \cdot \delta \bm{v}_T)-\delta \bm{h} [\delta \bm{h}_T]^2-\delta \bm{h} [\delta \bm{v}_T(\bl)]^2\B] d^3 \bl \\&+3 \lim_{\varep\to0}\int_{\mathbb{T}^{3}} \f{1}{|\bl|}\varphi^\varep\bn\cdot\B[\delta \bm{v}\times(\delta \bm{h} \times\delta \bm{v})\B] d^3 \bl \\
     =&\f34\int_0^\infty r^3\varphi^{'}(r)dr 4\pi \bar {S}_{CHL}(\bm{v}, \bm{h})+\f34 \int_0^\infty 2\varphi(r) r^2 dr 4\pi \bar{S}_{CHT}(\bm{v}, \bm{h})\\&+3\int_0^\infty r^2\varphi(r) dr 4\pi \bar{S}_{CH}(\bm{v}, \bm{h})\\
=&-\f94 \bar {S}_{CHL}(\bm{v}, \bm{h})+\f32 \bar{S}_{CHT}(\bm{v}, \bm{h})+3\bar{S}_{CH}(\bm{v}, \bm{h}).
\end{aligned}\ee
Combining this  and \eqref{4.10}, we infer that, as $\varep\to0$,
\be\label{4.11}\left\{\ba
&D_{CH}(\bv,\bm{h})=-\f{15}{8} \bar{S}_{CHT}(\bm{v}, \bm{h})-\f32\bar{S}_{CH}(\bm{v}, \bm{h}),\\
&D_{CH}(\bv,\bm{h})=-\f94 \bar {S}_{CHL}(\bm{v}, \bm{h})+\f32 \bar{S}_{CHT}(\bm{v}, \bm{h})+3\bar{S}_{CH}(\bm{v}, \bm{h}),
 \ea\right.\ee
 which imples
\be\label{4.12}\ba
D_{CH}(\bv,\bm{h})=&-\f54\bar{S}_{CHL}(\bm{v}, \bm{h})+\bar{S}_{CH}(\bm{v}, \bm{h})\\
=&-\f54\B[\bar{S}_{CHL}(\bm{v}, \bm{h})-\f45\bar{S}_{CH}(\bm{v}, \bm{h})\B]\\
=&-\f54 S_{CHL}(\bm{v}, \bm{h}).
\ea\ee
 Hence, according to \eqref{4.10} and \eqref{4.12}, we deduce the 4/5 law and 8/15 law of cross-helicity in magnetied fluids
$$S_{C H L}(\bm{v}, \bm{h})=-\f{4}{5}D_{CH}(\bm{v}, \bm{h}),~\text{and}~S_{CHT}(\bm{v}, \bm{h})=-\f{8}{15}D_{CH}(\bm{v}, \bm{h}),$$
where
\be\ba
&S_{CHL}(\bm{v}, \bm{h})\\=&\lim_{\lambda \to0}\f{1}{\lambda}\int_{\partial B } \bn \cdot \B[ 2\delta\bm{v}(\lambda\bl) \big(\delta\bm{h}_L(\lambda\bl)\cdot \delta\bm{v}_L(\lambda\bl)\big)- \delta \bm{h}(\lambda\bl) [\delta \bm{v}_L(\lambda\bl)]^2- \delta \bm{h} (\lambda\bl)[\delta \bm{h}_L(\lambda\bl)]^2\B] \f{d\sigma(\bl) }{4\pi}\\
&-\f45\lim_{\lambda \to0}\f{1}{\lambda}\int_{\partial B } \bn \cdot\B[\delta \bm{v}(\lambda\bl)\times(\delta\bm{h}(\lambda\bl)\times \delta\bm{v}(\lambda\bl))\B] \f{d\sigma(\bl) }{4\pi},
\ea\ee
and
\be\ba
&S_{CHT}(\bm{v}, \bm{h})\\=& \lim_{\lambda \to0}\f{1}{\lambda}\int_{\partial B } \bn \cdot \B[ 2\delta\bm{v}(\lambda\bl) \big(\delta\bm{h}_T(\lambda\bl)\cdot \delta\bm{v}_T(\lambda\bl)\big)- \delta \bm{h}(\lambda\bl) [\delta \bm{v}_T(\lambda\bl)]^2- \delta \bm{h} (\lambda\bl)[\delta \bm{h}_T(\lambda\bl)]^2\B] \f{d\sigma(\bl) }{4\pi}\\
&+\f45\lim_{\lambda \to0}\f{1}{\lambda}\int_{\partial B } \bn \cdot\B[\delta \bm{v}(\lambda\bl)\times(\delta\bm{h}(\lambda\bl)\times \delta\bm{v}(\lambda\bl))\B] \f{d\sigma(\bl) }{4\pi}.
\ea\ee
The proof of Theorem 1.2 is finished.
\end{proof}
\section{Conclusion}

One accurate result of the well-known K41 theory is four-fifths law  of energy   $$\langle[\delta \bv_{L}(\bm{r})]^{3} \rangle=-\f45\epsilon \bm{r},$$
 for  incompressible flow in turbulence. Another famous  scaling law of energy is Yaglom four-thirds type
$$\langle \delta \bv_{L}(\bm{r})[\delta \bv (\bm{r})]^{2}  \rangle=-\f43\epsilon \bm{r}
$$in incompressible fluid. Helicity is the
second quadratic conserved quantity in the Euler equations. Its 4/3 law is governed by
 $$
 \langle\delta \bv_{L}({\bm r})[\delta \bv(\bm{r})\cdot\delta \bomega(\bm{r})] \rangle -\f12\langle\delta \bomega_{L}({\bm r})[\delta \bv({\bm r})]^{2}  \rangle=-\f43\epsilon_{H} \bm{r},
 $$
which can be  found in \cite{[GPP]}. Moreover, there also exists 2/15 law of helicity (see \cite{[Chkhetiani]}). A natural question is that weather  Kolmogorov type law is valid for the helicity. We consider this issue and invoke Eyink's  velocity decomposition in \cite{[Eyink1]} to establish  the local 4/5 law and 8/15 law for the helicity. In contrast to classical Kolmogorov law and
 Eyink's work \cite{[Eyink1]}, we have to handle  with the interaction between the velocity and vorticity.
 To overcome this difficulty, we first try to find the gap between $\f{\partial}{\partial_{\ell_k}}(\bn_i\bn_j)$ and $(\f{\partial \bn_i}{\partial\ell_j}+\f{\partial \bn_j}{\partial\ell_i})\bn_k$ involving different variable quantities (see \eqref{2.27}).
More  precisely, a new identity
\eqref{c26} is observed, which allows us to deal with the  interaction of different physical quantities. Finally, we deduce the exact law
$$
S_{HL} (\bv,\bomega)= -\f45 D_{H}(\bv,\bomega),$$
which corresponds to
$$\langle \delta \bv_{L} (\delta \bv_{L} \cdot\delta \bomega_{L} )  \rangle- \f12\langle \delta \bomega_{L} (\delta \bv_{L} )^{ 2}  \rangle  -\f25\langle  \delta\bomega_{L}(\delta\bv)^{2}\rangle + \f25\langle  \delta\bv_{L}(\delta\bv\cdot\delta\bomega)\rangle =-\f45\epsilon_{H} \bm{r}.$$
The observation \eqref{c26} helps us to revisit the
  4/5 laws due to Politano and  Pouquet   in \cite{[PP2]}
$$\ba
 \langle[\delta \bv_{L}({\bm r})]^{3} \rangle- 6\langle {\bm h}^{2}_{L}\delta \bv_{L}({\bm r})  \rangle=-\f45\epsilon_{E} \bm{r}, ~~\langle[\delta {\bm h }_{L}({\bm r})]^{3} \rangle- 6\langle {\bm h}^{2}_{L}\delta \bv_{T}({\bm r})  \rangle=-\f45\epsilon_{C} \bm{r},
\ea
$$ for the   total energy and  cross-helicity in
  magnetized fluids, where the velocity field couples the magnetic field.   The scaling law   plays an important role  in   plasma turbulence. We refer the reader to \cite{[AB],[MS]} for recent review in this direction.
 Here, we obtain that
  $$S_{EL}(\bv, \bm{h})=-\f45 D_{E}(\bv, \bm{h}), ~~
 S_{C L}(\bv, \bm{h})=-\f45 D_{CH}(\bv, \bm{h}),$$
 which correspond  to
 $$\ba
&\langle \delta \bv_{L}  (\delta \bv_{L})^{2}  \rangle+  \langle \delta \bv_{L} (\delta \bm{h}_{L} )^{ 2}  \rangle  -2\langle \delta \bm{h}_{L} (\delta \bm{h}_{L} \cdot\delta \bv_{L} )
\rangle-\f45\langle \delta \bm{h}_{L}(\delta \bm{h} \cdot\delta \bm{v} )  \rangle+\f45\langle \delta \bm{v}_{L}(\delta \bm{h} )^{2}  \rangle=-\f45\epsilon_{E} \bm{r},\\
&2\langle \delta \bv_{L} (\delta \bm{h}_{L} \cdot\delta \bv_{L} )
\rangle- \langle \delta \bm{h}_{L} (\delta \bm{h}_{L} )^{2}   \rangle-\langle \delta \bm{h}_{L} (\delta \bv_{L})^{2}\rangle-\f45\langle \delta \bm{h}_{L}  (\delta \bv )^{2}
\rangle+\f45\langle \delta \bm{v}_{L} ( \delta \bv\cdot \delta \bm{h}  )
\rangle=-\f45\epsilon_{C} \bm{r}.
\ea$$
It should be pointed out that an analogue of the above  four-fifths laws also hold for the inviscid MHD equations based on the
 Els\"asser variances. It is an interesting question to deduce these exact laws via the corresponding K\'arm\'an-Howarth type equations directly.

In summary, three new four-fifths laws for the helicity and, total energy, cross-helicity in the   incompressible Euler and inviscid MHD equations are deduced, respectively. %To the best of our knowledge, they are not mentioned in previous works.
% There exists some potential  application of these laws in turbulence.

\section*{Acknowledgement}

Ye was partially sponsored by National Natural Science Foundation of China (No. 11701145)
and Natural Science Foundation of Henan (No. 232300420111). Wang was partially supported by the National Natural Science Foundation of China under
grant (No. 11971446 and No. 12071113) and sponsored by Natural Science Foundation of
Henan (No. 232300421077).

\end{document}